\documentclass[11pt]{amsart}
\usepackage{amsmath,amsfonts,amssymb,amsthm,a4wide,amsbsy,vaucanson-g}
\usepackage[utf8]{inputenc}

\usepackage{graphicx}

\DeclareMathOperator{\rep}{rep}
\DeclareMathOperator{\Rec}{Rec}
\DeclareMathOperator{\Def}{Def}
\DeclareMathOperator{\val}{val}

\DeclareMathOperator{\dom}{dom}
\DeclareMathOperator{\sign}{sign}

\theoremstyle{plain}
\newtheorem{theorem}{Theorem}
\newtheorem{corollary}[theorem]{Corollary}
\newtheorem{prop}[theorem]{Proposition}
\newtheorem{fact}[theorem]{Fact}
\newtheorem{facts}[theorem]{Facts}
\newtheorem{lemma}[theorem]{Lemma}
\newtheorem{conjecture}[theorem]{Conjecture}
\theoremstyle{definition}
\newtheorem{definition}[theorem]{Definition}
\newtheorem{remark}[theorem]{Remark}
\newtheorem{example}[theorem]{Example}
\newtheorem*{conj}{Conjecture}
\newtheorem*{thm}{Theorem}

\begin{document}
\title[An analogue of Cobham's theorem for GDIFS]{An analogue of Cobham's theorem for graph directed iterated function systems}

\author{\'Emilie CHARLIER}
\author{Julien LEROY}
\author{Michel RIGO}
\address{Universit\'e de Li\`ege, Institut de math\'ematique, Grande traverse 12 (B37),
4000 Li\`ege, Belgium\newline
echarlier@ulg.ac.be, M.Rigo@ulg.ac.be}
\address{Mathematics Research Unit, FSTC, University of Luxembourg, 6, rue Coudenhove-Kalergi, 
L-1359 Luxembourg, Luxembourg\newline
julien.leroy@uni.lu}

\begin{abstract}
Feng and Wang showed that two homogeneous iterated function systems in $\mathbb{R}$ 
with multiplicatively independent contraction ratios necessarily have different attractors. 
In this paper, we extend this result to graph directed iterated function systems in $\mathbb{R}^n$ with contraction ratios 
that are of the form $\frac{1}{\beta}$, for integers $\beta$.
By using a result of Boigelot {\em et al.}, this allows us to give a proof of a conjecture of Adamczewski and Bell.
In doing so, we link the graph directed iterated function systems to Büchi automata.
In particular, this link extends to real numbers $\beta$.
We introduce a logical formalism that permits to characterize sets of $\mathbb{R}^n$ 
whose representations in base $\beta$ are recognized by some Büchi automata. 
This result depends on the algebraic properties of the base: $\beta$ being a Pisot or a Parry number.
The main motivation of this work is to draw a general picture representing 
the different frameworks where an analogue of Cobham's theorem is known. 
\end{abstract}

\maketitle

\section{Introduction}

According to Hutchinson~\cite{Hutchinson:1981}, any iterated function system (IFS) $\Phi = \{\phi_i\}_{i=1}^t$ 
admits a unique compact set $K$ (called the {\em attractor}) such that $K = \phi_1(K) \cup \cdots\cup \phi_t(K)$. 
Such a set is said to be {\em self-similar} or {\em self-affine} when the contraction mappings are affine functions.
For instance, it is well known that the usual Triadic Cantor set is self-affine for $\phi_1(x)=x/3$ and $\phi_2(x)=x/3+2/3$.

IFSs play an important role in fractal geometry, notably thanks to the so-called Collage theorem stating 
that any subset of a complete metric space can be approximated (with arbitrary precision) by a self-similar set~\cite{BD:1985}. 

Given an IFS, a natural question is whether its attractor can be obtained as the attractor of another IFS. 
This question has recently been completely answered by Feng and Wang~\cite{Feng&Wang:2009} 
in the case of homogeneous self-affine sets in $\mathbb{R}$ 
(all contraction mappings have the same contraction ratio) that satisfy the {\em open set condition} (OSC).

\begin{thm}[Feng and Wang~\cite{Feng&Wang:2009}]
Let $X$ be the attractor of a homogeneous IFS $\Phi = \{\phi_i\}_{i=1}^N$ in the Euclidean space $\mathbb{R}$ 
with contraction ratio $r_\Phi$ and satisfying the OSC. Consider $\psi(x) = \lambda x +d$.
\begin{enumerate}
	\item	Suppose that the Hausdorff dimension of X, $\dim_H(X)$, is less than $1$ and that $\lambda \neq 0$. 
			If $\psi(X) \subset X $, then $\frac{\log|\lambda|}{\log|r_\Phi|} \in \mathbb{Q}$;
	\item	Suppose $\dim_H(X)=1$ and $X$ is not a finite union of intervals and $\lambda > 0$. 
			If $\min(X) \in \psi(X) \subset X$, then $\frac{\log \lambda}{\log|r_\Phi|} \in \mathbb{Q}$.
\end{enumerate}
\end{thm}

Under a more restrictive separation condition (the so-called {\em strong separation condition}), 
Elekes, Keleti and M\'athé~\cite{EKM:2010} proved an extension of this result for self-similar sets in higher dimension 
and without the homogeneous hypothesis.  
In their setting, if $r_1,\dots,r_t$ denote the different contraction ratios of the IFS and if $\psi$ 
is a similarity with contraction ratio $r_\psi$ such that $\psi(X) \subset X$, 
then $\log |r_\psi|$ is a linear combination of $\log |r_1|$, \dots, $\log |r_t|$ with rational coefficients.

A way to understand these results is the following: If a ``non-simple" set $X \subset \mathbb{R}$ 
(meaning $X$ is not a finite union of intervals) is the attractor of two different IFSs, 
then the contraction ratios of these two IFSs have to be strongly linked.

Surprisingly enough, a rather similar result has recently been proved by Adamczewski and Bell~\cite{Adamczewski&Bell:2011}. 
If $X$ is a compact set in $[0,1]$ and if $b\geq 2$ is an integer, the {\em $b$-kernel} of $X$ 
is the collection of sets $(b^k X - a) \cap [0,1]$ for $a,k \in \mathbb{N}$, $0 \leq a < b^k$.
Such a set $X$ is said to {\em $b$-self-similar} if its $b$-kernel is finite.

\begin{thm}[Adamczewski and Bell~\cite{Adamczewski&Bell:2011}]
Let $b,b' \geq 2$ be two integers such that $\frac{\log b}{\log b'} \notin \mathbb{Q}$. 
A compact set  $X \subset [0,1]$ is simultaneously $b$- and $b'$-self-similar 
if and only if it is a finite union of intervals with rational endpoints.
\end{thm}

They conjectured an equivalent result in higher dimension.

\begin{conj}[Adamczewski and Bell~\cite{Adamczewski&Bell:2011}]
Let $b,b' \geq 2$ be two integers such that $\frac{\log b}{\log b'} \notin \mathbb{Q}$. 
A compact set $X \subset [0,1]^n$ is simultaneously $b$- and $b'$-self-similar 
if and only if it is a finite union of polyhedra whose vertices have rational coordinates.
\end{conj}


Even more surprisingly, a third similar result having an important meaning in the framework of IFSs has recently been proved, 
but in a theoretical computer science 
setting~\cite{Boigelot&Brusten&Bruyere:2008, Boigelot&Brusten&Leroux:2009,Boigelot&Brusten:2009}: 
It concerns sets of points in $\mathbb{R}^n$ whose representations in some integer base $b$ 
are accepted by some particular Büchi automaton. 
In this paper we provide a bridge between these three results. 
This can be achieved by use of graph directed iterated functions systems (GDIFS) and this allows us to provide extensions 
of the results in the three frameworks, for instance by proving Adamcweski and Bell's conjecture 
(which has also independently been proved by Chan and Hare~\cite{hare2013}) 
and by extending the IFS results to a large class of GDIFS.
Furthermore, we also extend a logical characterization of recognizable sets in some non integer bases, 
thus providing a fourth framework to study them.

Let us give a few more details about the result of Boigelot {\em et al.}.
A {\em B\"uchi automaton} is a finite state automaton with an accepting procedure adapted to infinite words~\cite{PerrinPin04}.
Given an integer base $b \geq 2$, one can therefore study subsets $X$ of $\mathbb{R}^n$ 
such that the base $b$ representations of the elements of $X$ are accepted by some Büchi automaton. 
Such a set is said to be {\em $b$-recognizable}.
These sets have been well studied and in particular, they
can be characterized by some first order formula in the structure
$\langle \mathbb{R},\mathbb{Z},+,\le \rangle$ extended with a special
predicate related to the chosen base \cite{BRW98}.
This logical characterization has for instance found applications in verification
of systems with unbounded mixed variables taking integer or real values. 
In that setting, timed automata or hybrid systems are considered
\cite{Boigelot&Bronne&Rassart:1997,Boigelot&Jodogne&Wolper:2005}.

{\em Weak Büchi automata} are a subclass of Büchi automata that mainly behaves like automata 
accepting finite words~\cite{Loding:2001}. 
Sets recognized by weak automata are said to be {\em weakly recognizable}.

\begin{thm}[Boigelot, Brusten, Bruyère, Jodogne, Leroux 
and Wolper~\cite{Boigelot&Jodogne&Wolper:2001,Boigelot&Brusten&Bruyere:2008,Boigelot&Brusten&Leroux:2009}]
Let $b,b' \geq 2$ be two integers such that $\frac{\log b}{\log b'} \notin \mathbb{Q}$. 
A set $X \subset \mathbb{R}^n$ is simultaneously weakly $b$- and $b'$-recognizable 
if and only if it is definable by a first order formula in the structure $\langle \mathbb{R}, \mathbb{Z}, +, \leq \rangle$.
\end{thm}

Observe that subsets of $\mathbb{R}^n$ that are definable by a first order formula in the structure 
$\langle \mathbb{R}, \mathbb{Z}, +, \leq \rangle$ are exactly the periodic repetitions of finite unions of polyhedra 
whose vertices have rational coordinates.

As already announced, in this paper, we obtain a better and complete understanding of the links 
and interactions existing between those results.
The different notions under consideration  are not exactly equivalent but very closely related. 
By means of GDIFSs, we are able to determine exactly what are the similarities and the differences between them.
 
GDIFSs generalize IFSs in the sense that the similarities can be applied accordingly to a labeled directed graph $(V,E)$~\cite{Edgar08}. 
As for IFSs, such a system admits a unique 
{\em attractor} which is a list of non-empty compact sets $(K_v)_{v \in V}$ such that for all $u \in V$,
\[
	K_u = \bigcup_{e: u \to v} S_e(K_v).
\]

Using these GDIFS, we obtain the following picture of inclusions, in the extended framework of some real base $\beta >1$.

\begin{figure}[h!tbp]
\centering
\begin{pspicture}(-3,-1)(10,5.5)
\pspolygon(-1,0.6)(-1,4)(5,4)(5,0.6)
\pspolygon(-1.2,0.4)(-1.2,4.2)(10,4.2)(10,0.4)
\rput(9,0.7){GDIFS}
\psellipse(2,2.1)(5,2.9)
\psellipse(6.3,2.3)(2.8,1)
\rput(8,2.3){IFS}
\rput(0.15,3.5){$\beta$-self-similar}
\rput(0,3){$\Updownarrow$}
\rput(2,3){Theorems~\ref{thm: digraph implique self-similar} and~\ref{thm:self-similar implique digraph}}
\rput(0.55,2.5){particular GDIFS}
\rput(0,2){$\Updownarrow$}
\rput(1.3,2){Theorem~\ref{thm:digraph automate}}
\rput(0.5,1.5){particular weakly}
\rput(0.25,1.1){$\beta$-recognizable}
\rput(2,0.1){$\beta$-recognizable $\Leftrightarrow$ $\beta$-definable}
\rput(2,-0.3){Theorems~\ref{the:recdef} and~\ref{thm: definissable implique reconnaissable}}
\end{pspicture}
\caption{General picture of our contributions.}
\label{fig: general picture}
\end{figure}
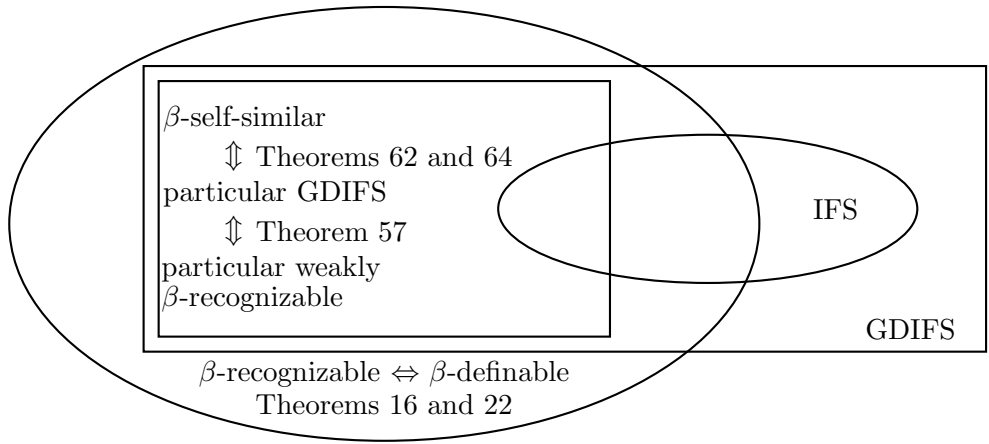

It is interesting to determine the strengths and weaknesses of the different approaches. 
On the one hand, the IFS method is very powerful from the point of view of contraction ratios: 
these can be any number smaller than 1 although we are restricted to reciprocal of positive integers in the automata framework.
Furthermore we can simultaneously work with different contraction ratios~\cite{EKM:2010}. 
From an automata point of view, this would correspond to represent numbers by simultaneously using several bases.
On the other hand, the homogeneous case of IFS corresponds to recognizable sets whose underlying automaton has a unique state.
In this direction, B\"uchi automata provide a more general result.

\bigskip

Observe that both the result of Adamwzewski and Bell and the one of Boigelot {\em et al.} 
belong to a large class of results that are generalisations of a famous theorem in discrete mathematics 
and theoretical computer science: Cobham's theorem. 
As in the real case described above, one can consider sets of integers 
whose base $b$ expansions are accepted by some finite automaton. 
These sets are also said to be {\em $b$-recognizable}. 
In 1969, Cobham obtained a fundamental result initiating the systematic study of $b$-recognizable sets 
and showing that the $b$-recognizability depends on the chosen base \cite{Cobham:1969}. 
Let $b,b'\ge 2$ be two integers such that $\frac{\log b}{\log b'} \notin \mathbb{Q}$. 
If a subset $X$ of $\mathbb{N}$ is simultaneously $b$-recognizable and $b'$-recognizable, 
then it is a finite union of arithmetic progressions.

Cobham's theorem led to a deep study, for instance by considering non-standard numeration systems: 
It is meaningful to consider other numeration systems to handle new sets of integers recognizable by finite automata. 
Making use of numeration systems like the Zeckendorf system based on
the Fibonacci sequence, other kinds of sets of integers can be
recognized by means of some finite automaton.
The bibliography in \cite{Durand:2011} provides many pointers to various extensions of Cobham's theorem.

In the setting of recognizable subsets of $\mathbb{R}^n$, as for the representation of integers 
where one has the opportunity to consider non-standard numeration systems, we are not limited to base $b$ expansions. 
We can consider other ways to represent real numbers.
A large class of extensively studied numeration systems is made of the so-called $\beta$-numeration
systems \cite[Chap.~7]{Lothaire:2002} and \cite{Parry60}, 
where the idea is to replace the integer base with a real base $\beta>1$.
In such a system, a real number usually can have more than one
representation as an infinite sequence and this redundancy can be
exploited in information theory with the use of beta encoders for analogues to digital conversion, 
see for instance \cite{Daubechies2006,Daubechies2010}.
The results mentioned in Figure~\ref{fig: general picture} depend on the algebraic properties of $\beta$ 
and Pisot numbers play a special role.

The paper is organized as follows. 
In Section~\ref{section: definitions} we recall the notions of $\beta$-expansions and of Büchi automata. 
In Section~\ref{section: rec def} and Section~\ref{section: def rec} we prove the equivalence 
between $\beta$-recognizable sets and $\beta$-definable sets 
(Theorems~\ref{the:recdef} and~\ref{thm: definissable implique reconnaissable}).
In Section~\ref{section: gdifs}, we prove the equivalence between $\beta$-self-similar sets, 
attractors of particular GDIFS and particular weakly $\beta$-recognizable sets 
(Theorems~\ref{thm:digraph automate}, \ref{thm: digraph implique self-similar} 
and~\ref{thm:self-similar implique digraph}).

\section{B{\"u}chi automata and $\beta$-expansions}
\label{section: definitions}
In this section, we describe the conventions that we will use to
represent real numbers or $n$-tuples of reals numbers using a real
base $\beta$. We recap some well-known results on B{\"u}chi automata.

Let $A$ be a finite alphabet. We denote the set of nonempty finite words over $A$ by $A^+$ 
and the set of infinite words by $A^\omega$.

\subsection{$\beta$-expansions}
Let $\beta>1$ be a real number and let $C \subset \mathbb{Z}$ be an
alphabet.  For a real number $x$, any infinite word $u = u_k \cdots
u_1 u_0 \star u_{-1} u_{-2} \cdots$ over $C \cup\{\star\}$ such that
\[
\val_\beta(u):= \sum_{-\infty < i \leq k} u_i \, \beta^i = x
\]  
is a {\em $\beta$-representation} of $x$.  In general, the
representation is not unique.  For $x \geq 0$, among all such
$\beta$-representations of $x$, we distinguish the {\em
  $\beta$-expansion} $d_\beta(x)=x_k\cdots x_1 x_0\star
x_{-1}x_{-2}\cdots$ which is an infinite word over the canonical
alphabet $A_\beta = \{0,\ldots,\lceil\beta\rceil -1\}$ containing one symbol $\star$ and obtained by a
greedy algorithm, i.e., we fix the minimal $k \geq 0$ such that
\[
x=\sum_{-\infty < i \leq k} x_i\, \beta^i \text{ and, for all }i \leq
k,\ x_i \in A_\beta
\] 
and, for all $\ell \le k$,
\[
\sum_{-\infty < i \leq \ell} x_i\, \beta^i < \beta^{\ell+1}.
\]
The choice of $k$ implies that $x_k \neq 0$ if and only if $x \geq 1$
and that reals in $[0,1)$ have a $\beta$-expansion of the form $0\star u$ with $u\in A_\beta^\omega$. 
In particular $d_\beta(0) = 0 \star 0^\omega$.  

To deal with negative real numbers and negative digits, we make use of
the notation $\overline{a}$ to denote the integer $-a$ for all $a\in\mathbb{Z}$.  
This notation is extended to a morphism acting on finite or infinite words over 
$\mathbb{Z}$: $\overline{u\, v}=\overline{u}\, \overline{v}$, $\overline{u\star v}
=\overline{u}\star\overline{v}$ and $\overline{\overline{u}}=u$.
Let $\bar{A}_\beta = \{\bar{0},\bar{1},\dots,\overline{\lceil \beta \rceil -1}\}$.  
For $x< 0$, the {\em $\beta$-expansion of $x$} is
defined as the word $d_\beta(x)=u \star v$ over $\bar{A}_\beta \cup \{\star\}$ 
such that $\overline{d_\beta(-x)} = u \star v$.  
With the identification $\bar{0} = 0$, we let $\tilde{A}_\beta$ denote the
alphabet $A_\beta \cup \bar{A}_\beta$.

\begin{definition}    
A real number $x$ is a {\em $\beta$-integer}, if $d_\beta(x)$ is of the kind $u \star 0^\omega$.  
The set of $\beta$-integers is denoted by $\mathbb{Z}_\beta$.
\end{definition}

Let $\beta >1$. The so-called {\em R\'enyi expansion} of $1$ plays a special role. 
It is the greatest word $w\in A_\beta^\omega$ (in the lexicographic order) 
not ending in $0^\omega$ such that $0\star w$ is a $\beta$-representation of $1$. 
It is denoted by $d_\beta^{\,*}(1)$.
The next result characterizes the $\beta$-expansions of real numbers
in $[0,1)$ in terms of their shifted words.  
Recall that the {\em shift map} is the function $\sigma : A^\omega \to A^\omega$ 
defined by $\sigma(u_1u_2u_3\cdots) = u_2u_3u_4\cdots$.

\begin{theorem}[Parry~\cite{Parry60}] \label{thm: parry}
Let $\beta >1$ be a real number. An infinite word $u$ is such that $0
\star u$ is the $\beta$-expansion of a real number in $[0,1)$ if and
only if for all $k \geq 0$, $\sigma^k(u) < d_\beta^{\, *}(1)$.
\end{theorem}

\begin{example}\label{ex:1}
For all $\beta>1$ one has $d_\beta(1)=1\star 0^\omega$.  
The definition of $d_\beta^*(1)$ indeed depends on $\beta$.  
Let $\varphi=(1+\sqrt{5})/2$ be the golden mean. 
We have $d_\varphi^*(1)=(10)^\omega$.  
Theorem~\ref{thm: parry} shows that the $\varphi$-expansions of real numbers in $[0,1)$ are of the
form $0 \star u$, where $u$ is an infinite word over $\{0,1\}$ not containing $11$ as a factor.
\end{example}

\subsection{B{\"u}chi automata and $\omega$-regular languages}

\begin{definition}
Given an alphabet $A$, a {\em B\"uchi automaton} over $A$ is a
labeled directed graph given by a $5$-tuple
$\mathcal{A}=(Q,A,E,I,T)$ where $Q$ is the finite set of states,
$E \subseteq Q \times A \times Q$ is the set of transitions, $I$
is the set of initial states and $T$ is the set of terminal states.
An infinite word $u \in A^\omega$ is said to be {\em accepted} by
$\mathcal{A}$ if there exists an infinite path in $\mathcal{A}$
whose label is $u$ starting in an initial state and visiting infinitely often the set $T$.  
The language $L \subseteq A^\omega$ accepted by $\mathcal{A}$ 
is the set of all infinite words over $A$ that are accepted by $\mathcal{A}$; 
such a language is said to be {\em $\omega$-regular}.  
A B\"uchi automaton is said to be {\em deterministic} if $I$ is a singleton and 
for all states $q \in Q$ and all letters $a \in A$, 
there is at most one transition in $\{q\} \times \{a\} \times Q$. 
As usual when representing an automaton, initial states are depicted with an incoming arrow and terminal states with an outgoing arrow.
\end{definition}

If $u=u_0\cdots u_{\ell-1},v=v_0\cdots v_{\ell-1}$ are two finite
words of the same length $\ell$ (resp. $u=u_0u_1\cdots ,v=v_0v_1\cdots $ are
two infinite words), then one can define the {\em direct product} $u\times v$ of
$u$ and $v$ where the $i$th symbol is $(u\times v)_i=(u_i,v_i)$ for all $i<\ell$ (resp. for
all $i\ge 0$). Let $L\subseteq A^\omega$ and $M\subseteq B^\omega$ be
two languages of infinite words, we define the {\em direct product} of $L$
and $M$ as $L\times M=\{u\times v \mid u\in L, v\in M\}$. The
morphisms of projections are denoted by $\pi_1$ and $\pi_2$, satisfy
$\pi_1(u\times v)=u$, $\pi_2(u\times v)=v$ and are naturally extended
to the product of languages. These notions of product and projection
can obviously be extended to the direct product of $n$ words.

\begin{facts}[\cite{PerrinPin04}] \label{fact: buchi stable} 
We collect some well-known result on $\omega$-regular languages.
\begin{enumerate}
  \item The class of $\omega$-regular languages is closed under
    complementation, finite union, finite intersection, morphic image and inverse image under a morphism.
	
  \item If $L$ and $M$ are $\omega$-regular languages over the
    alphabets $A$ and $B$ respectively, then $L \times M$ is an $\omega$-regular language over $A \times B$.
	
  \item If $L$ is an $\omega$-regular language over an alphabet $A
    \times B$, then the projections $\pi_1(L)$ and $\pi_2(L)$ are $\omega$-regular languages.
	
  \item The class of languages accepted by deterministic B\"uchi
    automaton is strictly included in the class of $\omega$-regular languages.
\end{enumerate}
\end{facts}

\begin{example}
Let us illustrate Fact~\ref{fact: buchi stable}~(4).  
The non-deterministic B\"uchi automaton in Figure~\ref{finitely-many-a} accepts the language over $\{a,b\}$
of the words containing finitely many $a$'s.  
However it can easily be shown that no deterministic B\"uchi automaton accepts the same language.
\begin{figure}[h!tbp]
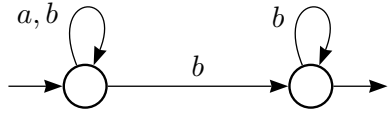

\centering
\begin{VCPicture}{(0,0)(3,2)}
\ChgStateLineWidth{.5}
\ChgEdgeLineWidth{.5}
\ChgEdgeLabelScale{.6}

\SmallState\State[]{(0,0.5)}{A}
\SmallState\State[]{(3,0.5)}{B}
\Initial{A}
\Final[e]{B}
\EdgeL[.5]{A}{B}{b}
\CLoopN{A}{a,b}
\CLoopN{B}{b}
\end{VCPicture}
\caption{A non-deterministic B\"uchi automaton}
\label{finitely-many-a}
\end{figure}
\end{example}

\subsection{Recognizable and definable sets of reals} \label{subsection: recognizable and definable}
There is no obstacle to consider a multi-dimensional framework. 
Let $n\ge 1$ be an integer.  Let us define the following three alphabets of $n$-tuples:
\[
{\bf A_\beta}=\underbrace{A_\beta \times \cdots \times A_\beta}_{n
  \text{ times}}, \quad {\bf \bar{A}_\beta}=\underbrace{\bar{A}_\beta
  \times \cdots \times \bar{A}_\beta}_{n \text{ times}} \quad \text{ and }\quad {\bf
  \tilde{A}_\beta}=\underbrace{\tilde{A}_\beta
  \times \cdots \times \tilde{A}_\beta }_{n \text{ times}}.
\] 
We also set 
\[
\boldsymbol{\star}=(\underbrace{\star,\ldots , \star}_{n  \text{ times}}) \ \text{ and } \ {\bf 0}
=(\underbrace{0,\dots,0}_{n  \text{ times}}).
\]

\begin{definition}
Let $\mathbf{x}=(x_1,\dots,x_n)$ be a point in $\mathbb{R}^n$.  
We define the $\beta$-expansion of $\mathbf{x}$ as being the word
$d_\beta(\mathbf{x})$ over the alphabet $ {\bf \tilde{A}_\beta}\cup
\{\boldsymbol{\star}\}$ that belongs to $0^* d_\beta(x_1) \times 0^*
d_\beta(x_2)\times \dots\times 0^*d_\beta(x_n)$ and that does not start
with ${\bf 0}$ except if $|x_i|<1$ for all $i$, 
in which case we consider the word starting with $\mathbf{0\star}$. 
Otherwise stated, the $n$ $\beta$-expansions are
synchronized by possibly using some leading zeroes in such a way that
all the $\star$ symbols occur at the same position in every $\beta$-expansions.
\end{definition}

\begin{example}
Take $\varphi=(1+\sqrt{5})/2$. Consider $\mathbf{x}=(x_1,x_2)=((1+\sqrt{5})/4,2+\sqrt{5})$. We have
\[
d_\varphi(\mathbf{x})=
\begin{array}{ccccccccccc}
    0&0&0& \star &1&0&0&1&0&0&\cdots\\
    1&0&1& \star &0&1&0&1&0&1&\cdots\\
\end{array}
\]
Where the first $\varphi$-expansion is padded with some leading zeroes.
Consider an example where all the components have moduli less than one. 
With $\mathbf{y}=(x_1,x_2)=((1+\sqrt{5})/4,-1/2)$, we get 
\[
d_\varphi(\mathbf{y})=
\begin{array}{ccccccccc}
    0& \star &1&0&0&1&0&0&\cdots\\    
    0& \star &0&\overline{1}&0&0&\overline{1}&0&\cdots\\
\end{array}
\]
where the two $\varphi$-expansions start with one symbol $0$ followed by $\star$.
\end{example}

\begin{definition}
A set $X \subseteq \mathbb{R}^n$ is {\em $\beta$-recognizable} 
if there is some B\"uchi automaton over the alphabet ${\bf \tilde{A}_\beta}\cup \{\boldsymbol{\star}\}$ 
accepting the language $d_\beta(X)$.  
We let $\Rec_\beta(\mathbb{R}^n)$ denote the set of $\beta$-recognizable sets of $\mathbb{R}^n$.
\end{definition}

\begin{fact}\label{fact: 3 ways to recognize}
Let $X \subseteq \mathbb{R}^n$. The following three assertions are equivalent:
\begin{enumerate}
  \item $X$ is $\beta$-recognizable;
  \item there exists some B\"uchi automaton over the alphabet 
	${\bf \tilde{A}_\beta}\cup \{\boldsymbol{\star}\}$ accepting the language ${\bf 0}^* d_\beta(X)$;
  \item there exists a B\"uchi automaton over the alphabet 
	${\bf \tilde{A}_\beta}\cup \{\boldsymbol{\star}\}$ accepting a language of the form 
	$\{{\bf 0}^{m(\mathbf{x})} d_\beta(\mathbf{x}) \mid \mathbf{x} \in X\}$ 
	for some map  $m:\mathbf{x}\to \mathbb{N}$.
\end{enumerate}
\end{fact}

A {\em Parry number} is a real number $\beta$ for which the R\'enyi expansion of $1$ is finite 
({\em simple Parry number}) or ultimately periodic ({\em non-simple Parry number}).  
The next result is a direct consequence of Theorem~\ref{thm: parry}.
The following fact provides a justification to restrict ourselves to Parry numbers.
It is indeed desirable that the set of all $\beta$-expansions is accepted by some (deterministic) Büchi automaton.

\begin{fact}\cite{Bertrand89,Parry60} \label{fact: parry recognizable}
If $\beta>1$ is a Parry number, then $d_\beta([0,1)^n)$ is accepted by a deterministic B\"uchi automaton.
\end{fact}

\begin{example}
The following B\"uchi automaton in Figure~\ref{figure:aut-Fibo} accepts $d_\varphi([0,1))$ where $\varphi=(1+\sqrt{5})/2$.
\begin{figure}[h!tbp]
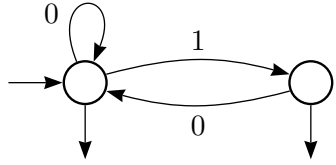

\centering
\begin{VCPicture}{(0,0)(3,2)}
\ChgStateLineWidth{.5}
\ChgEdgeLineWidth{.5}
\ChgEdgeLabelScale{.6}
\SmallState\State[]{(0,1)}{A}
\SmallState\State[]{(3,1)}{B}
\Initial{A}
\Final[s]{A}
\Final[s]{B}
\ArcL[.5]{A}{B}{1}
\ArcL[.5]{B}{A}{0}
\CLoopN{A}{0}
\end{VCPicture}
\caption{A deterministic B\"uchi automaton accepting $d_\varphi([0,1))$.}
\label{figure:aut-Fibo}
\end{figure}
\end{example}

The next definition introduces a predicate describing that some fixed
digit occurs in a specific position of the $\beta$-expansion of a
given real number (assuming that $\beta$-expansions are extended to
the left with infinitely many zeroes).
\begin{definition}
    Let $x$ be a real number and let $d_\beta(x)=x_k\cdots x_0\star
    x_{-1} \cdots$ be its $\beta$-expansion.  For each $a\in
    \tilde{A_\beta}$, the binary predicate $X_{\beta,a}(x,y)$ over
    $\mathbb{R}^2$ holds true whenever $y$ is an integral power of $\beta$ and either
    \begin{itemize}
      \item $|x|<y$ and $a=0$, or
      \item $|x|\ge y$ and in the $\beta$-expansion of $x$, the coefficient corresponding to $y$ is $a$, i.e., 
		$y=\beta^i$ for some $i\le k$ and $x_i=a$.
    \end{itemize}
    Let $X_\beta$ denote the (finite) collection of predicates $\{X_{\beta,a}\mid a\in \tilde{A_\beta}\}$.  
    A set $X\subseteq\mathbb{R}^n$ is {\em $\beta$-definable} if it can be
    defined by a first order formula in the structure
    \[
	\langle \mathbb{R}, 1, \le, +, X_\beta \rangle.
    \] 
We let $\Def_\beta(\mathbb{R}^n)$ denote the set of $\beta$-definable sets of $\mathbb{R}^n$.
\end{definition}

\begin{remark}
The property of being a power of $\beta$ is definable in the structure 
$\langle \mathbb{R}, 1, \le, X_\beta\rangle$ by the formula
\[
	x \text{ is a power of }\beta \Leftrightarrow (\exists y)\left(X_{\beta,1}(x,y) \wedge x=y\right).
\]
We can also define the properties of being a positive or negative power 
of $\beta$ by adding $x > 1$ or $x < 1$ respectively. Let $b$ be a power of $\beta$. 
One can easily define the next (or the previous) power of $\beta$ as follows:
\[
b'=\beta b\Leftrightarrow 
(b'\text{ is a power of }\beta) \wedge (b'>b) \wedge (\forall
c)((c\text{ is a power of }\beta \wedge c>b)\implies c\ge b').
\] 
Consequently, any constant (positive or negative) power of $\beta$ is definable in the structure.
\end{remark}

\begin{remark}
The two structures $\langle \mathbb{R}, 1, \le, +,X_\beta\rangle$ and 
$\langle \mathbb{R}, \mathbb{Z}_\beta, \le, +, X_\beta\rangle$ are equivalent. 
Indeed, the set $\mathbb{Z}_\beta$ can be defined in $\langle \mathbb{R}, 1, \le, X_\beta\rangle$ by the formula
\[
	z \in \mathbb{Z}_\beta 
	\Leftrightarrow 
	(\forall y)
	\big[ (y  \text{ is a negative power of } \beta)
	\implies
	X_{\beta,0}(z,y)
	\big].
\]
Conversely, $1$ can be defined in $\langle \mathbb{R}, \mathbb{Z}_\beta, \le,+\rangle$ by the formula
\[
	z = 1 \Leftrightarrow (z \in \mathbb{Z}_\beta) \wedge \big[\big((x \in \mathbb{Z}_\beta) \wedge
	(x > 0) \big) \implies (x \geq z) \big]
\]
where $0$ is defined by $(\forall x)(x+0 = x)$.
\end{remark}

\section{If the base is a Parry number, recognizability implies definability}
\label{section: rec def}

The main result of this section is the following one. Compared with the next section, 
note that the only assumption is that $\beta$ is a Parry number.

\begin{theorem}\label{the:recdef}
    Let $\beta>1$ be a Parry number.  If $X\subseteq\mathbb{R}^n$ is
    $\beta$-recognizable, then $X$ is $\beta$-definable.
\end{theorem}

The aim of the following technical lemma is to ensure the construction
of some valid $\beta$-expansions.

\begin{lemma}\label{lem:tech}
    Let $\beta>1$ be a Parry number. 
    There exists an integer $k$  such that for any infinite word
    $u\in\{0,1\}^\omega$ of the form $0^{n_1}10^{n_2}10^{n_3}\cdots$
    with $n_i\ge k$ for all $i\ge 1$, there exists some real $x$ such that $d_\beta(x)=0\star u$.
\end{lemma}

\begin{proof}
Since $\beta$ is a Parry number, there exist two finite words $u$ and $v$ such that $d_\beta^*(1)=uv^\omega$ 
where $v$ is nonempty and contains at least one non-zero digit. 
Taking $k=|u|+|v|$, the result follows from Theorem~\ref{thm: parry}.
\end{proof}

In the proof of Theorem~\ref{the:recdef}, we will make use of the following construction. 
We can replace one automaton by $M$ copies of this automaton in such a
way that any path cyclically visits these copies. 
The reason of this construction is to use a number of copies corresponding to the
constant $k$ obtained in the previous lemma.  Let $M\ge 1$ be an integer. 
If $\mathcal{A}=(Q,A,E,I,T)$ is a B\"uchi automaton, 
we define the automaton $\mathcal{A}^{(M)}=(Q',A,E',I',T')$ as follows.
Its set of states is made of $M$ distinct copies of the states of $\mathcal{A}$: 
If $Q=\{q_1,\ldots,q_t\}$ then $Q'=\cup_{n=1}^M\{q_{1,n},\ldots,q_{t,n}\}$. 
If $(q_i,a,q_j)$ belongs to $E$, then $(q_{i,n},a,q_{j,n+1})\in E'$ for all $n\in
\{1,\ldots,M-1\}$ and $(q_{i,M},a,q_{j,1})\in E'$. 
The set of initial states $I'$ is made of the sates $q_{i,1}$ such that $q_i\in I$. 
A state $q_{i,n}$, $1\le n\le M$, is terminal whenever $q_i\in T$.

\begin{example} 
Consider the automaton $\mathcal{A}$ depicted in Figure~\ref{figure:aut-Fibo}. 
We consider the automaton $\mathcal{A}^{(3)}$ depicted in Figure~\ref{figure:aut-Fibo3}. 
All states are terminal but for the sake of readability, outgoing arrows have been omitted. 
The three copies have been drawn consecutively from left to right.
\begin{figure}[h!tbp]
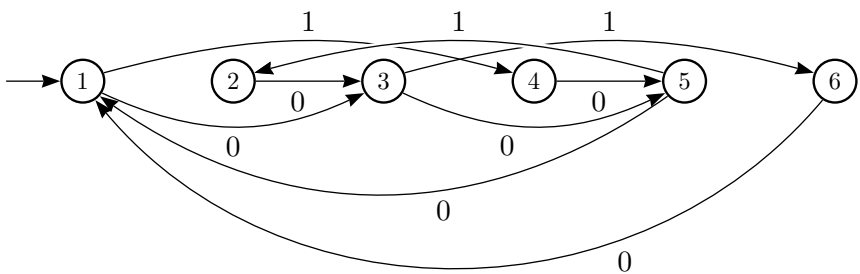

\centering
\begin{VCPicture}{(0,-1.7)(10,2)}
\ChgStateLineWidth{.5}
\ChgEdgeLineWidth{.5}
\ChgEdgeLabelScale{.6}
\ChgStateLabelScale{.5}

\SmallState\State[1]{(0,1)}{A}
\SmallState\State[2]{(2,1)}{B}
\SmallState\State[3]{(4,1)}{A2}
\SmallState\State[4]{(6,1)}{B2}
\SmallState\State[5]{(8,1)}{A3}
\SmallState\State[6]{(10,1)}{B3}
\Initial{A}

\ArcL[.5]{A}{B2}{1}
\EdgeR[.4]{B}{A2}{0}
\ArcL[.5]{A2}{B3}{1}
\EdgeR[.4]{B2}{A3}{0}
\EdgeBorder
\ArcR[.5]{A3}{B}{1}
\EdgeBorderOff
\VArcR[.5]{arcangle=-25}{A}{A2}{0}
\VArcR{arcangle=-25}{A2}{A3}{0}
\VArcL{arcangle=35}{A3}{A}{0}
\VArcL[.3]{arcangle=50}{B3}{A}{0}
\end{VCPicture}
\caption{An automaton of the kind $\mathcal{A}^{(3)}$.}
\label{figure:aut-Fibo3}
\end{figure}
\end{example}

\begin{lemma}
Let $\mathcal{A}$ be a B\"uchi automaton. 
The automaton $\mathcal{A}^{(M)}$ accepts the same $\omega$-regular language.
\end{lemma}

\begin{proof}
Clear from the definition of $\mathcal{A}^{(M)}$.
\end{proof}

For the proof of Theorem~\ref{the:recdef}, we follow essentially the
same lines as \cite{BRW98}. We code runs (i.e., infinite sequences of
states) in a B\"uchi automaton $\mathcal{A}$ using a finite number of
infinite words over $A_\beta\cup\{\star\}$. 
These infinite words in turn can be viewed as $\beta$-expansions of some real numbers. 
Hence, this coding permits to replace runs with real numbers. 
Successful computations in $\mathcal{A}$ can therefore be translated into a
formula in $\langle \mathbb{R},1,\le,+,X_\beta\rangle$. The main
technical difference when dealing with $\beta$-expansions instead of
classical base $b$ expansions is that we have to ensure that the
infinite words that are built are valid $\beta$-expansions (cf. Lemma~\ref{lem:tech}).

We will make often use of the following fact. 

\begin{remark}\label{rem20}
Multiplication (or division) by $\beta$ and thus by a constant power of $\beta$ is also definable:
\[
y=\beta x \Leftrightarrow 
(\forall b) 
[\bigwedge_{a\in \tilde{A}_\beta} (X_{\beta,a}(x,b)\implies X_{\beta,a}(y,\beta b))].
\]
This formula  expresses that we get the $\beta$-expansion of $y$ simply by shifting the one of $x$.
\end{remark}

\begin{proof}[Proof of Theorem~\ref{the:recdef}]
    To keep notations as simple as possible, without loss of
    generality, we may assume that $n=1$. By assumptions, there exists
    a B\"uchi automaton $\mathcal{A}$ with $t$ states accepting the
    language $d_\beta(X)$.  Moreover, we can also assume that it has a
    single initial state by possibly adding a new initial state with
     convenient transitions.

    Let $k$ be given by Lemma~\ref{lem:tech}. We replace $\mathcal{A}$
    with the automaton $\mathcal{A}^{(k+1)}=(Q,\tilde{A}_\beta\cup\{\star\},E,I,T)$
    having $s:=(k+1)t$ states. We can enumerate its states as
    $Q=\{q_1,\ldots,q_{s}\}$ and assume that $q_1$ is the initial
    state of $\mathcal{A}^{(k+1)}$. Each state $q_j$ will be coded by
    a unit column vector in $\{0,1\}^{s}$ where each component is equal to
    zero except for the $j$th component. This vector is denoted by $c(q_j)$.

    To any infinite run $r=(q_{i_j})_{j\ge 0}\in Q^\omega$ in
    $\mathcal{A}^{(k+1)}$, it corresponds an infinite sequence of
    vectors $\{0,1\}^{s}$ coding the sequence of visited states. If we
    concatenate these column vectors, for each row, we get $s$
    infinite words $w_1,\ldots,w_s$ over $\{0,1\}$ with the property
    that each symbol $1$ is followed by at least $k$ zeroes. Indeed,
    the shortest cycles in $\mathcal{A}^{(k+1)}$ have length at least
    $k+1$ and therefore, the same unit vector cannot be encountered
    more than once every $k+1$ times. As an example, consider the
    automaton $\mathcal{A}^{(3)}$ depicted in
    Figure~\ref{figure:aut-Fibo3}. One of the shortest loops is given
    by $1\stackrel{0}{\longrightarrow} 3\stackrel{0}{\longrightarrow}
    5\stackrel{0}{\longrightarrow} 1$ and it is derived from the loop
    with label $0$ in the original automaton depicted in
    Figure~\ref{figure:aut-Fibo}. The periodic run $(1,3,5)^\omega$
    corresponds to the $6$ infinite words 
\[
\begin{array}{rccccccccc}
w_1&=&1&0&0&1&0&0&\cdots\\
w_2&=&0&0&0&0&0&0&\cdots\\
w_3&=&0&1&0&0&1&0&\cdots\\
w_4&=&0&0&0&0&0&0&\cdots\\
w_5&=&0&0&1&0&0&1&\cdots\\
w_6&=&0&0&0&0&0&0&\cdots\\
\end{array}
\]

Thanks to Lemma~\ref{lem:tech}, the $s$ infinite
    words $w_1,\ldots,w_s$ over $\{0,1\}$ are such that $0\star w_1,\ldots, 0\star w_s$ are valid
    $\beta$-expansions. We define the map $$f:Q^\omega\to\mathbb{R}^s,\ 
    r=(q_{i_j})_{j\ge 0}\mapsto (\val_\beta(0\star w_1),\ldots, \val_\beta(0\star w_s)).$$
In our running example, 
$f(r)=((1+\sqrt{5})/4,0,1/2,0,1/(1+\sqrt{5}),0)$.

    We now define a $(s+1)$-ary predicate  $R_{\mathcal{A}}(x,y_1,\ldots,y_s)$ over
    $\mathbb{R}^{s+1}$ which holds true if and only if there
    exists an execution of $\mathcal{A}^{(k+1)}$ with the infinite word
    $d_\beta(x)$ producing a successful run $r$ such that $f(r)=(y_1,\ldots,y_s)$.

    Recall that a successful run must start in the initial state, be
    compatible with the transitions of $\mathcal{A}^{(k+1)}$ and visit
    infinitely often the set of terminal states. Our final task is to
    show that $R_{\mathcal{A}}(x,y_1,\ldots,y_s)$ can be expressed in $\langle
    \mathbb{R},1,\le,+,X_\beta\rangle$. Therefore the set $X$ of reals
    accepted by $\mathcal{A}$ is $\beta$-definable:
\[
\{x\in\mathbb{R}\mid (\exists
y_1)\ldots(\exists y_s)(R_{\mathcal{A}}(x,y_1,\ldots,y_s))\}.
\]

    We first overcome some technicalities to make some kind of
    synchronization between the expansion of $x$ and the expansions of
    $y_1,\ldots,y_s$. Indeed, we have access through the predicates
    $X_{\beta,\cdot}(\cdot,b)$ to the digits corresponding to a
    same power $b$ of $\beta$ in the $\beta$-expansions of $x$ and
    $y_1,\ldots,y_s$. The sequence of visited states is encoded in
    $y_1,\ldots,y_s$ and we have to synchronize the state reached at
    step $n$ with the symbol read at the same step $n$ and the state
    at step $n+1$.  We assume in all what follows that $x$ is positive
    (the final formula should involve a disjunction of the two
    possible cases to take into account the sign of $x$). 
    In the next formula which is a part of the definition of 
    $R_{\mathcal{A}}(x,y_1,\ldots,y_s)$, we define $s$ new intermediate
    variables $z_1,\ldots,z_s$ that are roughly shifted versions of $y_1,\ldots,y_s$:
\begin{eqnarray*}
(\exists b)[(b\text{ is a power of }\beta)\wedge (b\le x< \beta b)\wedge 
 (((x\ge 1)\wedge \bigwedge_{i=1}^s (z_i=\beta b y_i))\vee ((x<1)\wedge \bigwedge_{i=1}^s (z_i=y_i)))].
\end{eqnarray*}
More precisely, if $x\ge 1$, then $d_\beta(x)=u\star v$ where $u$ is a
finite nonempty word not starting with $0$. We set
$d_\beta(z_i)=u_i\star v_i$ for all $i$. Then at least one of the
$u_i$'s is non reduced to $0$ and the longest $u_i$ as the same
length as $u$. If $x<1$, then $d_\beta(z_i)=0 \star v_i$ for all $i$
and one of the $v_i$ has $1$ as prefix. 

To give the reader an idea
about this construction, we continue our running example with
$(y_1,\ldots,y_s)=((1+\sqrt{5})/4,0,1/2,0,1/(1+\sqrt{5}),0)$ 
coding a sequence of states and taking $x=\varphi^3$. 
In such a case, we get $z_i=\varphi^4 y_i$, for all $i$,
and the following $\varphi$-expansions:
\[
\begin{array}{rcccccccccc}
  d_\varphi(x)&=&1&0&0&0&\star&0&0&\cdots\\    
d_\varphi(z_1)&=&1&0&0&1&\star&0&0&\cdots\\
d_\varphi(z_2)&=&0&0&0&0&\star&0&0&\cdots\\
d_\varphi(z_3)&=&0&1&0&0&\star&1&0&\cdots\\
d_\varphi(z_4)&=&0&0&0&0&\star&0&0&\cdots\\
d_\varphi(z_5)&=&0&0&1&0&\star&0&1&\cdots\\
d_\varphi(z_6)&=&0&0&0&0&\star&0&0&\cdots\\
\end{array}.
\]
The leftmost column contains the coding of the first state of the sequence of states but also the first symbol that is read. 
Generally, the $n$th column contains the coding of the $n$th state and the $n$th read symbol. 
Note that if $|x|<1$, then the leftmost column is $(0,\ldots,0)$
followed directly with $(\star,\ldots,\star)$.

The predicate
    $R_{\mathcal{A}}(x,y_1,\ldots,y_s)$ is true if and only if
    \begin{enumerate}
      \item the run starts in the initial state: 
     \begin{multline*}
\hspace{1cm} (\exists b_0)\big[
X_{\beta,1}(z_1,b_0) \wedge X_{\beta,0}(z_2,b_0)\wedge \cdots \wedge X_{\beta,0}(z_s,b_0) \\
\wedge (\forall c)[c>b_0\implies
(X_{\beta,0}(z_1,c) \wedge\cdots \wedge X_{\beta,0}(z_s,c))]\big].
\end{multline*}
Note that $b_0$ is defined once and for all and corresponds to the largest power of $\beta$ 
occurring with a non-zero coefficients in one of the $\beta$-developments of $z_1,\ldots,z_s$.

\item The run follows the transitions: for all powers $c$ of $\beta$
  less or equal to $b_0$, being in the state coded by the digits
  corresponding to $c$ occurring in $d_\beta(z_1),\ldots,d_\beta(z_s)$
  and reading the digit corresponding to $c$ occurring in
  $d_\beta(x)$, the reached state must correspond to the digits
  corresponding to $c/\beta$ (we make use of Remark~\ref{rem20}) occurring in
  $d_\beta(z_1),\ldots,d_\beta(z_s)$. Such a finite transition relation can
  be coded by a formula:
\begin{multline*}
\hspace{1cm} (\forall c)
\big[ c\le b_0 \implies 
\big\{\bigwedge_{\substack{(q,d,q')\in E\\ c(q)=(a_1,\ldots,a_s)\\ c(q')=(a_1',\ldots,a_s')}}  
X_{\beta,a_1}(z_1,c)\wedge \cdots \wedge X_{\beta,a_s}(z_s,c) \wedge X_{\beta,d}(x,c)\\
\hspace{8cm} \implies X_{\beta,a_1'}(z_1,c/\beta)\wedge \cdots \wedge X_{\beta,a_s'}(z_s,c\beta)\big\}\big]. \\
\end{multline*}

\item Finally, the run must visit infinitely often a terminal state: for
  all powers $c$ of $\beta$ less or equal to $b_0$, there exists a
  power $d<c$ of $\beta$ such that the state coded by the digits
  corresponding to $d$ occurring in $d_\beta(z_1),\ldots,d_\beta(z_s)$
  is terminal: 
\begin{multline*}
(\forall c)\big [(c \text{ is a power of }\beta)\wedge (c\le b_0) \implies 
(\exists d)
 (d< c)  \wedge 
\bigvee_{\substack{q\in T\\ c(q)=(a_1,\ldots,a_s)}} 
X_{\beta,a_1}(z_1,d)\wedge \cdots \wedge X_{\beta,a_s}(z_s,d)\big].\\
\end{multline*}
    \end{enumerate}
\end{proof}

\section{If the base is a Pisot number, definability implies recognizability}
\label{section: def rec}

A real algebraic integer greater than $1$ whose Galois conjugates have
modulus less than $1$ is called a {\em Pisot number}. 
Under this stronger assumption, we get the converse of Theorem~\ref{the:recdef}.

\begin{fact}\cite[Chapter 7]{Lothaire:2002}
If $\beta$ is a Pisot number, then it is a Parry number. But the converse does not hold. 
\end{fact}

\begin{theorem}
\label{thm: definissable implique reconnaissable}
Let $\beta>1$ be a Pisot number. If $X \subseteq \mathbb{R}^n$ is $\beta$-definable, then $X$ is $\beta$-recognizable.
\end{theorem}

To prove this result, we use the classical method consisting in proving the recognizability of any formula by induction, 
i.e., we prove that sets defined by atomic formulae are recognizable 
and that adding connectors and quantifiers does not alter recognizability. 
We first need a few results. 

\begin{lemma}
\label{lemma: R^n recognizable}
If $\beta > 1$ is a Parry number, then $\mathbb{R}^n$ is $\beta$-recognizable.
\end{lemma}

\begin{proof} 
Let $\mathcal{A}_{\beta,{\rm frac}} = (Q_{\rm frac}, {\bf A}_\beta\cup\{\boldsymbol{\star}\}, E_{\rm frac} , I_{\rm frac} , 
T_{\rm frac})$ denote a deterministic B\"uchi automaton 
accepting the language $d_\beta([0,1)^n)$ (see Fact~\ref{fact: parry recognizable}). 
Thus the language accepted by $\mathcal{A}_{\beta,{\rm frac}}$ is included 
in ${\bf 0} \boldsymbol{\star} {\bf A_\beta^\omega}$. 
Let $\mathcal{A}'$ be a deterministic B\"uchi automaton over the alphabet ${\bf A}_\beta$ such that 
$L(\mathcal{A}_{\beta,{\rm frac}}) = {\bf 0} \boldsymbol{\star} L(\mathcal{A}')$.
Starting from $\mathcal{A}'$, we will show how to build an automaton $\mathcal{A}_\beta$ 
that accepts $\mathbf{0}^*d_\beta(\mathbb{R}^n)$.
We then conclude by using Fact~\ref{fact: 3 ways to recognize}.

If $r\in \mathbb{R}$ we define $\sign(r)$ to be $+$ if $r\ge 0$ and $-$ else.
If ${\bf x}=(x_1,\ldots,x_n)$ is a point in $\mathbb{R}^n$, then $\sign({\bf x})=(\sign(x_1),\ldots,\sign(x_n))$.
Given an $n$-tuple ${\bf z}=(z_1,\ldots,z_n)$ whose components belong to $\{+,-\}$, 
we will build a B\"uchi automaton $\mathcal{A}_{\beta,{\bf z}}$ accepting 
$\mathbf{0}^*d_\beta(\{{\bf x} \in\mathbb{R}^n  \mid \sign({\bf x})={\bf z}\})$. 
Then the  B\"uchi automaton $\mathcal{A}_\beta$ accepting $\mathbf{0}^*d_\beta(\mathbb{R}^n)$ 
will be the (disjoint) union of these $2^n$ automata. 

We construct such an automaton $\mathcal{A}_{\beta,{\bf z}}$ by considering two copies of $\mathcal{A}'$, 
one for the $\beta$-integer part and one for the $\beta$-fractional part of the representations. 
In all labels of transitions of both copies of  $\mathcal{A}'$, 
we replace the $i$-th component  by its ``opposite value" if $\sign(z_i)=-$ and we leave it unchanged otherwise.

W.l.o.g. we can pick the automaton $\mathcal{A}'$ so that its initial state has a loop labeled by ${\bf 0}$.
The unique initial state of $\mathcal{A}_{\beta,{\bf z}}$ is a new additional state $i$
and, for each transition $(p,\mathbf{a},q)$ with $ \mathbf{a} \in {\bf \tilde{A}_\beta}$ 
and  $p,q$ states of the $\beta$-integer part copy of $\mathcal{A}'$ with $p$ initial,
there is a transition $(i,\mathbf{a},q)$ in $\mathcal{A}_{\beta,{\bf z}}$. 
We add a loop on the initial state $i$ with label ${\mathbf 0}$.

The terminal states are the terminal states of the $\beta$-fractional part copy. 
We complete $\mathcal{A}_{\beta,{\bf z}}$ by adding, for each state $q$ of $\mathcal{A}'$, 
a transition from $(q,{\rm int})$ to $(q,{\rm frac})$ labeled by $\boldsymbol{\star}$, 
where $(q,{\rm int})$ (resp. $(q,{\rm frac})$) is the state of $\mathcal{A}_{\beta,{\bf z}}$ 
that corresponds to $q$ in the $\beta$-integer part copy (resp. $\beta$-fractional part copy).
\end{proof}

\begin{remark}
We illustrate the proof of the previous lemma. The automaton $\mathcal{A}_{\varphi,+}$ is depicted in Figure~\ref{figure:a+}.
\begin{figure}[h!tbp]
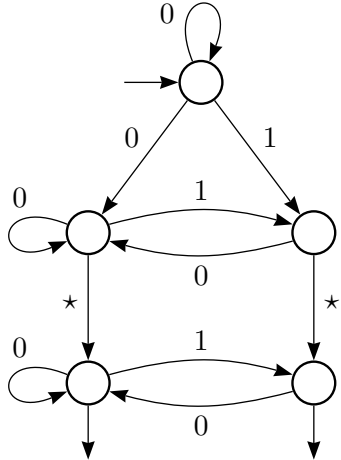

\centering
\begin{VCPicture}{(0,-2)(3,4)}
\ChgStateLineWidth{.5}
\ChgEdgeLineWidth{.5}
\ChgEdgeLabelScale{.6}
\SmallState\State[]{(1.5,3)}{I}
\SmallState\State[]{(0,1)}{A}
\SmallState\State[]{(3,1)}{B}
\SmallState\State[]{(0,-1)}{C}
\SmallState\State[]{(3,-1)}{D}
\Initial{I}
\Final[s]{C}
\Final[s]{D}
\ArcL[.5]{A}{B}{1}
\ArcL[.5]{B}{A}{0}
\CLoopW{A}{0}
\ArcL[.5]{C}{D}{1}
\ArcL[.5]{D}{C}{0}
\CLoopW{C}{0}
\EdgeR{A}{C}{\star}
\EdgeL{B}{D}{\star}
\EdgeR{I}{A}{0}
\EdgeL{I}{B}{1}
\CLoopN{I}{0}
\end{VCPicture}
\caption{The automaton $\mathcal{A}_{\varphi,+}$.}
\label{figure:a+}
\end{figure}
\end{remark}

\begin{lemma}
\label{lemma: Z^n rec}
If $\beta > 1$ is a Parry number, then $\mathbb{Z}_\beta^n$ is $\beta$-recognizable.
\end{lemma}

\begin{proof}
Indeed, the automaton recognizing $\mathbb{Z}_\beta^n$ is simply the intersection 
of the one recognizing $\mathbb{R}^n$ with the one accepting 
${\bf \tilde{A}_\beta^+} \boldsymbol{\star} {\bf 0}^\omega$.
\end{proof}

\begin{lemma}
If $\beta > 1$ is a Parry number, then the set $X_< = \{(x,y) \in \mathbb{R}^2 \mid x < y\}$ is $\beta$-recognizable.
\end{lemma}

\begin{proof}
For all $x,y \in \mathbb{R}$, we have $x<y$ if and only if $d_\beta(x) \prec d_\beta(y)$ 
where $\prec$ is a natural generalization of the lexicographic order with negative digits: 
it is defined as the partial order over $\tilde{A}_\beta^+ \star \tilde{A}_\beta^\omega$ by, 
for $u = u_k \cdots u_1 u_0 \star u_{-1} u_{-2} \cdots$ 
and $v = v_\ell \cdots v_1 v_0 \star v_{-1} v_{-2} \cdots$, 
$u \prec v$ if and only if
\begin{eqnarray*}
	&		& 	\big[
				\big(u \in \bar{A}_\beta^+ \star \bar{A}_\beta^\omega\big)  \wedge 
				\big(v \in A_\beta^+ \star A_\beta^\omega\big)  \wedge 
				\big( u,v\not \in 0^+\star 0^\omega \big)
				\big]	\\
	& \vee 	& 	\big[
				\big(u,v \in A_\beta^+ \star A_\beta^\omega\big) \wedge (k<\ell) 
				\big] \\
	& \vee 	& 	\big[
				\big(u,v \in \bar{A}_\beta^+ \star \bar{A}_\beta^\omega\big) \wedge (k>\ell) 
				\big] \\
	& \vee 	& 	\big[
				\big(	 (u,v \in A_\beta^+ \star A_\beta^\omega) 
				\vee (u,v \in \bar{A}_\beta^+ \star \bar{A}_\beta^\omega) \big) 
				\wedge  (k=\ell)	\\
	&	&		\quad \quad \wedge \big((\exists i \leq k) [(\forall j \in \{i+1,\dots,k\})( u_j = v_j) \wedge (u_i < v_i) ] \big)
				\big].
\end{eqnarray*}

Thus, the set $X_<$ is recognized by the intersection of the automaton accepting $d_\beta\left(\mathbb{R}^2\right)$ with the automaton represented in Figure~\ref{figure: automate <}.
\begin{figure}[h!tbp]
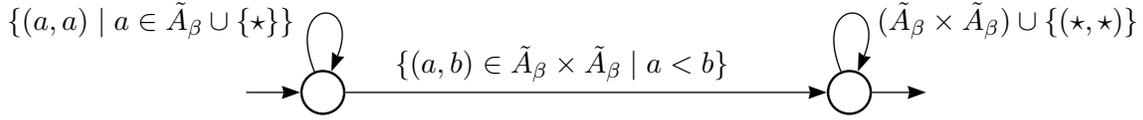

\centering
\begin{VCPicture}{(0,0)(7,2)}
\ChgStateLineWidth{.5}
\ChgEdgeLineWidth{.5}
\ChgEdgeLabelScale{.6}
\SmallState\State[]{(0,0.5)}{A}
\SmallState\State[]{(7,0.5)}{B}
\Initial{A}
\Final{B}
\EdgeL{A}{B}{\{(a,b) \in \tilde{A}_\beta\times\tilde{A}_\beta \mid a < b\}}
\CLoopN{A}{\{(a,a) \mid a \in \tilde{A}_\beta\cup \{\star\}\}}
\CLoopN[.75]{B}{(\tilde{A}_\beta \times \tilde{A}_\beta) \cup\{(\star,\star)\}}
\end{VCPicture}
\caption{Automaton for the order.}
\label{figure: automate <}
\end{figure}
\end{proof}

\begin{definition}
Let $\beta > 1$ be a real number and let $C \subset \mathbb{Z}$ be an alphabet. 
The {\em normalization function} is the function 
\[
	\nu_{\beta,C}	: C^+ \star C^{\omega} \to \tilde{A}_\beta^+ \star \tilde{A}_\beta^{\omega}
\]
that maps any $\beta$-representation of a real number $x$ onto its $\beta$-expansion $d_\beta(x)$.
\end{definition}

\begin{definition}
Given two alphabets $A$ and $B$, a {\em transducer} is a B\"uchi automaton 
given by a $6$-tuple $\mathcal{T} = (Q, A, B , E, I, T)$, 
where the edges are labeled by elements in $A^* \times B^*$ instead of considering the unique alphabet $A$. 
Thus a transducer defines a relation $R_\mathcal{T} \subseteq A^\omega \times B^\omega$ 
given by $(u,v) \in R_\mathcal{T}$ if and only if $(u,v)$ is accepted by $\mathcal{T}$. 
A transducer is said to be {\em letter-to-letter} if its edges are labeled by elements of $A \times B$. 
If $L \subseteq A^\omega$ is a language and $\mathcal{T}$ is a transducer, 
we let $\mathcal{T}(L)$ denote the language over $B$ such that $L \times \mathcal{T}(L) \subseteq R_\mathcal{T}$ 
and which is maximal (with respect to the inclusion) for this property.
\end{definition}

Frougny~\cite{Frougny92} studied the normalization $\nu_{\beta,C,{\rm frac}} : C^{\omega} \to A_\beta^{\omega}$ 
that maps any word $u \in C^\omega$ such that $\val_\beta(0 \star u) \in [0,1)$ onto the word $v \in A_\beta^\omega$ 
such that $d_\beta(\val_\beta(0\star u)) = 0\star v$. 
She proved the following result.

\begin{theorem}[Frougny~\cite{Frougny92}]\label{the:Frougny}
Let $\beta$ be a Pisot number and $C \subset \mathbb{Z}$ be an alphabet. 
The normalization $\nu_{\beta,C,{\rm frac}}$ is realizable by a (non-deterministic) letter-to-letter transducer 
$\mathcal{T}_{\beta,C,{\rm frac}} = ( Q, C, A_\beta, E, \{i\}, T)$:
For every infinite word $u \in C^\omega$ such that $\val_\beta(0\star u)\in[0,1)$, there exists a unique infinite word $v \in A_\beta^\omega$
such that $(u,v) \in R_{\mathcal{T}_{\beta,C,{\rm frac}}}$ and, moreover, $d_\beta(\val_\beta(0\star u)) = 0\star v$. 
\end{theorem}

\begin{example}
Let $\varphi$ be the Golden ratio $(1+\sqrt{5})/2$. Frougny's transducer $\mathcal{T}_{\varphi,C,{\rm frac}}$ 
for the alphabet $C=\{\bar{1},0,1\}$ is the composition of the transducer $\mathcal{T}$ 
depicted in Figure~\ref{transducteur-Frougny} and the one realizing the intersection 
with the set $D(\varphi)=\{w \in A_\varphi^\omega \mid \exists x\in[0,1)\  d_\varphi(x)=0\star w\}
=\{w\in \{0,1\}^\omega\mid w \text{ does not contain } 11 \text{ as a factor}\}$. 
In Figure~\ref{transducteur-Frougny} all states are considered terminal.
To obtain the transducer $\mathcal{T}_{\varphi,C,{\rm frac}}$, 
we make the product of $\mathcal{T}$ and the automaton in Figure~\ref{figure:aut-Fibo}:
each state $q$ of the transducer $\mathcal{T}$ is split into 2 states $(q,0)$ and $(q,1)$ 
in order to maintain the information whether 
the last  output letter  is $0$ or $1$.
If it was $1$, then it  is forbidden to output another $1$ at the next step. More precisely, 
$(p,i)\xrightarrow{a|b}(q,j)$ is an edge if and only if $p\xrightarrow{a|b}q$ is an edge in $\mathcal{T}$ 
and $(i,b,j)$ equals $(0,0,0)$, $(0,1,1)$ or $(1,0,0)$.
\begin{figure}[h!tbp]
\centering
\begin{VCPicture}{(-5,-7.5)(8,4.5)}
\ChgStateLineWidth{.5}
\ChgEdgeLineWidth{.5}
\ChgEdgeLabelScale{.6}
\ChgStateLabelScale{.5}
\SmallState
\StateVar[-\varphi]{(-3.5,3)}{A}
\StateVar[-1]{(-3.5,0)}{B}
\StateVar[0]{(0,0)}{C}
\StateVar[1]{(3.5,0)}{D}
\StateVar[-\varphi+1]{(-3.5,-3.5)}{E}
\StateVar[\varphi-2]{(-3.5,-6.5)}{F}
\StateVar[\varphi-1]{(3.5,-3.5)}{G}
\StateVar[\varphi]{(7,0)}{H}
\StateVar[\varphi+1]{(3.5,3)}{I}
\StateVar[2\varphi]{(0,3)}{J}
\StateVar[-\varphi+2]{(0,-6.5)}{K}
\StateVar[2\varphi-1]{(7,3)}{L}

\Initial[s]{C}

\small
\CLoopN{A}{1|0}

\EdgeR[.6]{B}{A}{0|0,1|1}
\VArcL[.5]{arcangle=10}{B}{E}{1|0}

\EdgeR{C}{B}{0|1, \bar{1}|0}
\CLoopN[.5]{C}{0|0,1|1}
\EdgeL{C}{D}{1|0}

\VArcL[.5]{arcangle=10}{D}{G}{0|1,\bar{1}|0}
\EdgeR{D}{H}{0|0,1|1}
\EdgeL{D}{I}{1|0}

\VArcL[.5]{arcangle=10}{E}{B}{0|0,1|1}

\VArcL[.75]{arcangle=35}{F}{A}{0|1,\bar{1}|0}
\EdgeR[.6]{F}{E}{0|0,1|1}
\VArcL[.5]{arcangle=10}{F}{K}{1|0}

\ArcL[.2]{G}{B}{\bar{1}|1}
\VArcL[.7]{arcangle=5}{G}{C}{0|1,\bar{1}|0}
\VArcL[.5]{arcangle=10}{G}{D}{0|0,1|1}

\CLoopE[.75]{H}{0|1,\bar{1}|0}
\EdgeL{H}{I}{0|0,1|1}
\ArcL[.7]{H}{G}{\bar{1}|1}

\EdgeR{I}{J}{0|1,\bar{1}|0}
\VArcL[.5]{arcangle=10}{I}{L}{\bar{1}|1}

\CLoopN{J}{\bar{1}|1}

\VArcR[.5]{arcangle=-40}{K}{H}{1|0}
\VArcL[.5]{arcangle=10}{K}{F}{0|1,\bar{1}|0}
\EdgeR{K}{G}{1|0}

\EdgeL{L}{H}{\bar{1}|1}
\VArcL[.4]{arcangle=10}{L}{I}{0|1,\bar{1}|0}

\EdgeBorder
\EdgeL[.7]{E}{C}{1|0}
\EdgeR[.7]{D}{F}{\bar{1}|1}
\EdgeBorderOff
\end{VCPicture}
\caption{The transducer $\mathcal{T}$.}
\label{transducteur-Frougny}
\end{figure}
Note that there is a transition in $\mathcal{T}$ of the form $r \xrightarrow{a|b} s$ whenever $\varphi r +a-b=s$.
In \cite[Proposition 2.3.38]{Frougny-CANT} this transducer is called the converter $\mathcal{C}_\varphi$.
\end{example}

\begin{remark}
\label{remark: transducer unicite de chemin}
The transducer $\mathcal{T}_{\beta,C,{\rm frac}}$ is built in such a way that for each state $q$ in $Q$, 
there is at most one pair of words $(0^k,u)$ such that 
$u$ does not begin with $0$ and that labels a path from $i$ to $q$.
\end{remark}

The next result extends Theorem~\ref{the:Frougny} to $\nu_{\beta,C}$. 
We first need to allow a transducer to have an initial function instead of initial states. 

\begin{definition}
A {\em transducer with an initial function} is a transducer given by a $6$-tuple 
$\mathcal{B} = (Q, A, B , E, \alpha, T)$, where $\alpha$ is a partial function over $Q$ with values in $A^* \times B^*$. 
A pair of infinite words $(u,v) \in A^\omega \times B^\omega$ 
is accepted by such a transducer if there is a state $q \in \dom(\alpha)$ 
and a pair of infinite words $(u',v') \in A^\omega \times B^\omega$ 
such that $(u,v) = \alpha(q) (u',v')$\footnote{Concatenation of letters in $A\times B$ 
is as follows: $(a,b)(c,d)=(ac,bd)$.} and $(u',v')$ labels an infinite path in $\mathcal{B}$ 
starting in $q$ and going infinitely often through terminal states.
\end{definition}

\begin{prop}
\label{proposition: normalization}
Let $\beta > 1$ be a Pisot number and let $C \subset \mathbb{Z}$ be an alphabet. 
The normalization $\nu_{\beta,C}$ 
is realizable by a  (non-deterministic) letter-to-letter transducer 
$\mathcal{T}_{\beta,C} = ( Q_{\beta,C}, C \cup \{\star\}, \tilde{A}_\beta\cup\{\star\}, E, \alpha_{\beta,C}, T_{\beta,C})$ 
with an initial function 
$\alpha_{\beta,C}: Q_{\beta,C} \to \{\varepsilon\} \times \tilde{A}_{\beta}^*$:
For every infinite word $u \in C^+ \star C^\omega$, there exists a unique infinite word 
$v\in(\tilde{A}_\beta\cup\{\star\})^\omega$
such that $(u,v) \in R_{\mathcal{T}_{\beta,C}}$ and, moreover, $v\in 0^* d_\beta(\val_\beta(u))$.
\end{prop}

\begin{proof}
Let us construct such a transducer $\mathcal{T}_{\beta,C}^+$ 
but only for words $u \in C^+ \star C^\omega$ such that $\val_\beta(u) \geq 0$. 
The case where $\val_\beta(u)<0$  is obtained by considering two copies of $\mathcal{T}_{\beta,C}^+$ where, 
in one of them, we have exchanged all labels by their respective opposite value.

As in the proof of Lemma~\ref{lemma: R^n recognizable}, we consider the transducer $\mathcal{T}$ 
that consists of two copies of $\mathcal{T}_{\beta,C,{\rm frac}}$, 
one for the $\beta$-integer part and one for the $\beta$-fractional part. 
The initial state of $\mathcal{T}$ is the initial state of the $\beta$-integer part copy 
and the terminal states of $\mathcal{T}$ are the terminal states of the $\beta$-fractional part copy. 
For any state $q$ of $\mathcal{T}_{\beta,C,{\rm frac}}$, we add a transition labeled 
by $(\star,\star)$ from $(q,{\rm int})$ to $(q,{\rm frac})$ where $(q,{\rm int})$ 
(resp. $(q,{\rm frac})$) is the state of $\mathcal{T}$ that corresponds to $q$ 
in the $\beta$-integer part copy (resp. $\beta$-fractional part copy).

By construction of $\mathcal{T}$, a pair of words 
$(u,v)\in (C^+ \star C^\omega)\times (\tilde{A}_\beta\cup\{\star\})^\omega$ 
is accepted by $\mathcal{T}$ if and only if 
$(u,v) \in (C \times A_\beta)^+ (\star,\star) (C \times A_\beta)^{\omega}$ 
and $v \in 0^* d_\beta(\val_\beta(u))$. 
Thus, what remains to consider is the case 
where $u = u_k \cdots u_0 \star u_{-1} u_{-2} \cdots \in C^+ \star C^{\omega}$ 
is such that $d_\beta(\val_\beta(u)) = v_\ell \cdots v_0 \star v_{-1} v_{-2} \in A_\beta^+ \star A_\beta^\omega$ with $\ell > k$.
The purpose of the initial function is to deal with this situation.

To be able to define such an initial function, it is first easily seen that in that case, the difference between $k$ and $\ell$ is bounded. 
Indeed, if $c = \max C$, we have 
\[
	\val_\beta (u) \leq \frac{c \beta^k}{\beta-1}.
\]
Note that we have $c\ge 1$ because $\val_\beta(u)\ge0$ and $\ell>k$.
The integer $\ell$ is the greatest one for which $\beta^{\ell} \leq \frac{c \beta^k}{\beta-1}$, 
which is equivalent to 
\[
	\ell -k \leq \left\lfloor \log_\beta\left( \frac{c}{\beta-1} \right) \right\rfloor =: K.
\]

Now let $u = u_k \cdots u_0 \star u_{-1} u_{-2} \cdots \in C^+ \star C^{\omega}$ 
be such that $d_\beta(\val_\beta(u)) 
= v_\ell \cdots v_0 \star v_{-1} v_{-2} \in A_\beta^+ \star A_\beta^\omega$ with $\ell > k$. 
By construction of $\mathcal{T}$, the pair of words 
$(0^{\ell-k}u,d_\beta(\val_\beta(u)))$ is accepted by $\mathcal{T}$. 
Thus, if $q$ is a state reached after reading $(0^{\ell-k},d_\beta(\val_\beta(u))[1\colon\ell-k])$,
we consider an initial function $\alpha$ defined by
$\alpha(q) = (\varepsilon,d_\beta(\val_\beta(u))[1\colon\ell-k])$.
Notice that Remark~\ref{remark: transducer unicite de chemin} 
implies that we can always define $\alpha$ in such a way.

Finally $\mathcal{T}_{\beta,C}^+$ is the transducer $\mathcal{T}$ to which we add the initial function $\alpha$. 
\end{proof}

The following result is folklore.
The proof is not very difficult but quite long and technical. 
Therefore we will only sketch the proof on an example.

\begin{lemma} \label{fact: transducer conserve la regularite}
If $\mathcal{T} = ( Q, A, B, E, \alpha, T )$ is a letter-to-letter transducer 
with partial initial function $\alpha: Q \to A^{\leq k} \times B^{\leq \ell}$ 
for some positive constants $k$ and $\ell$, then the relation $R_{\mathcal T}$ 
realized by $\mathcal T$ is an $\omega$-regular language over the alphabet $A\times B$.
In particular, if $L$ is an $\omega$-regular language over an alphabet $A$ 
then  $\mathcal{T}(L)$ is an $\omega$-regular language over $B$.
\end{lemma}

\begin{proof}[Sketch of the proof]
Consider the transducer $\mathcal T$ of Figure~\ref{fig:example-transducer}.
\begin{figure}[h!tbp]
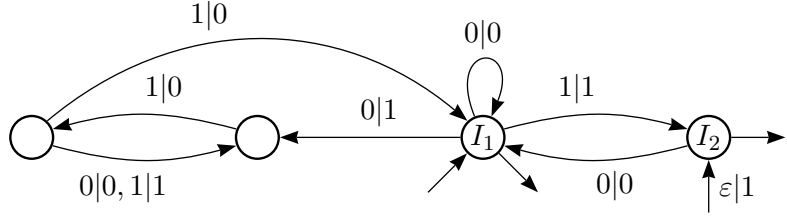

\centering
\begin{VCPicture}{(0,0)(9,3)}
\ChgStateLineWidth{.5}
\ChgEdgeLineWidth{.5}
\ChgEdgeLabelScale{.6}
\ChgStateLabelScale{.6}
\SmallState\State[]{(0,1)}{A}
\SmallState\State[]{(3,1)}{B}
\SmallState\State[I_1]{(6,1)}{C}
\SmallState\State[I_2]{(9,1)}{D}
\HideState
\SmallState\State[]{(9,-0.3)}{E}
\Initial[sw]{C}
\Final[se]{C}
\Final[e]{D}
\VArcL{arcangle=40}{A}{C}{1|0}
\ArcR{A}{B}{0|0,1|1}
\ArcR{B}{A}{1|0}
\EdgeR{C}{B}{0|1}
\ArcL{C}{D}{1|1}
\ArcL{D}{C}{0|0}
\CLoopN[0.5]{C}{0|0}
\EdgeR{E}{D}{\varepsilon|1}
\end{VCPicture}
\caption{Transducer with partial initial function.}
\label{fig:example-transducer}
\end{figure}
We build a new transducer consisting of two modified copies of $\mathcal T$, one for each initial state.
The copy associated with $I_1$ is unchanged, except that $I_2$ is no longer initial. 
In the copy associated with $I_2$, we remove the initial function from $I_2$ 
and then we must take into account the generated shift caused by the incoming digit $1$.
From $I_2$ we must write $1$ instead of $0$.
Each state is duplicated if one can write $0$ or $1$ when entering this state and we record this information 
by labeling each new state by $0$ or $1$. 
Else we label the state by the unique digit that one can write when entering the state.
Each new state has the same incoming and outgoing transitions as the state it comes from.
The first component of the labels of the transitions are unchanged 
while the second components are the labels of the outgoing corresponding state. 
The transducer that we obtain (see Figure~\ref{fig:resultat}) 
\begin{figure}[h!tbp]
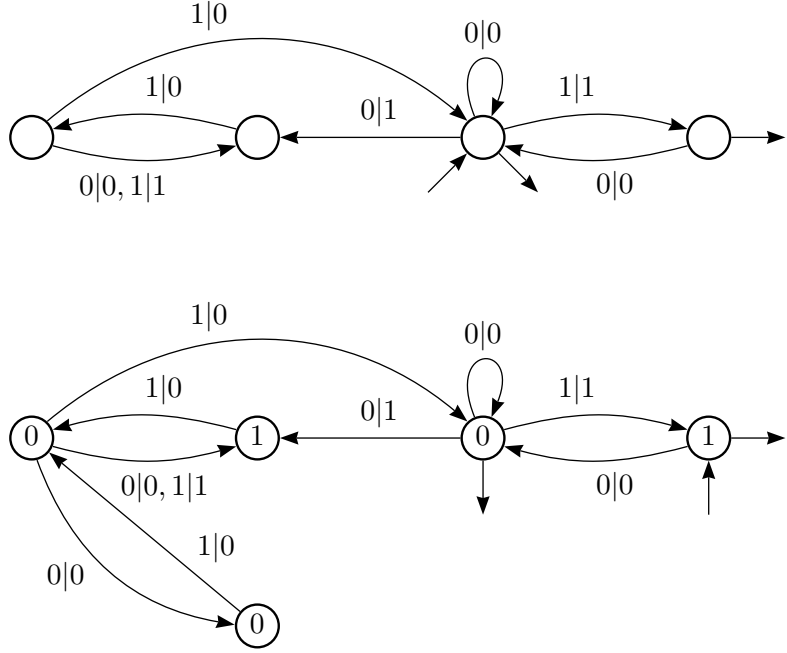

\centering
\begin{VCPicture}{(0,1)(9,10)}
\ChgStateLineWidth{.5}
\ChgEdgeLineWidth{.5}
\ChgEdgeLabelScale{.6}
\ChgStateLabelScale{.6}
\SmallPicture 
\SmallState\State[]{(0,8)}{A}
\SmallState\State[]{(3,8)}{B}
\SmallState\State[]{(6,8)}{C}
\SmallState\State[]{(9,8)}{D}
\Initial[sw]{C}
\Final[se]{C}
\Final[e]{D}
\VArcL{arcangle=40}{A}{C}{1|0}
\ArcR{A}{B}{0|0,1|1}
\ArcR{B}{A}{1|0}
\EdgeR{C}{B}{0|1}
\ArcL{C}{D}{1|1}
\ArcL{D}{C}{0|0}
\CLoopN[0.5]{C}{0|0}
\SmallState\State[0]{(0,4)}{EE}
\SmallState\State[1]{(3,4)}{F}
\SmallState\State[0]{(6,4)}{G}
\SmallState\State[1]{(9,4)}{H} 
\SmallState\State[0]{(3,1.5)}{FF} 
\Initial[s]{H}
\Final[s]{G}
\Final[e]{H}
\VArcL{arcangle=40}{EE}{G}{1|0}
\ArcR[0.6]{EE}{F}{0|0,1|1}
\ArcR{F}{EE}{1|0}
\EdgeR{G}{F}{0|1}
\ArcL{G}{H}{1|1}
\ArcL{H}{G}{0|0}
\CLoopN[0.5]{G}{0|0}
\LArcR{EE}{FF}{0|0}
\EdgeR[0.25]{FF}{EE}{1|0}
\end{VCPicture}
\caption{Modified transducer without a partial initial function.}
\label{fig:resultat}
\end{figure}
and the transducer $\mathcal T$ both recognize the same language.

If, for a state $q$, the initial function $\alpha$ is such that $\alpha(q)=(\varepsilon,u_1\cdots u_n)$ 
where $n\ge2$ and $u_i$ are digits, 
then in the copy corresponding to $q$, the new initial state $q'$ is such that $\alpha'(q')=(\varepsilon,u_1\cdots u_{n-1})$.
Then we iterate this process until the second component of the initial function is empty.
\end{proof}

\begin{corollary} \label{corollary: + rec}
Let $X,Y\subset \mathbb{R}$.
If $X$ and $Y$ are $\beta$-recognizable, then so is $X+Y$.
\end{corollary}

\begin{proof}
Let $\mathcal{A}_X$ and $\mathcal{A}_Y$ be the B\"uchi automata that respectively 
accept $d_\beta(X)$ and $d_\beta(Y)$.
We build another automaton $\mathcal{A}_{X \times Y}$ that accepts $d_\beta(X \times Y)$ by intersecting 
an automaton accepting $0^*d_\beta(X) \times 0^*d_\beta(Y)$ (see Fact~\ref{fact: buchi stable})
with an automaton accepting the words that contain exactly one occurrence of the letter $(\star,\star)$.
Now, let us consider the automaton $\mathcal{A}_+$ obtained from $\mathcal{A}_{X \times Y}$ 
by replacing each label $(a,b)$ by $a+b$ and replacing $(\star,\star)$ by $\star$. 
By construction, the automaton $\mathcal{A}_+$ accepts a language $L$ over the alphabet 
$B_\beta = \{0,1,2,\dots, 2(\lceil \beta \rceil-1), \bar{1},\bar{2},\dots,\overline{2(\lceil \beta \rceil-1)}\}$ 
such that $\val_\beta(L) = X+Y$. 
Finally, using the transducer $\mathcal{T}_{\beta,B_\beta}$ of Proposition~\ref{proposition: normalization}, 
Lemma~\ref{fact: transducer conserve la regularite} and Fact~\ref{fact: 3 ways to recognize} 
imply that $X+Y$ is $\beta$-recognizable.
\end{proof}

\begin{lemma} \label{lemma: X_beta rec}
For all $a \in \tilde{A}_\beta$, the set 
$X_a = \{(x,y) \in \mathbb{R}^2 \mid X_{\beta,a}(x,y) \text{ is true}\}$ is $\beta$-recognizable.
\end{lemma}

\begin{proof}
Indeed, $d_\beta(X_a)$ is accepted by the intersection of 
the automaton recognizing $d_\beta(\mathbb{R}^2)$ with the one represented in Figure~\ref{figure: X_a}.
\begin{figure}[h!tbp]
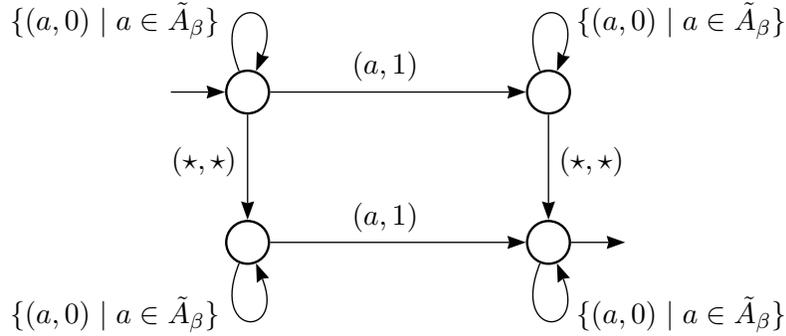

\centering
\begin{VCPicture}{(0,-0.5)(4,4.5)}
\ChgStateLineWidth{.5}
\ChgEdgeLineWidth{.5}
\ChgEdgeLabelScale{.6}
\SmallState\State[]{(0,3)}{A}
\SmallState\State[]{(4,3)}{B}
\SmallState\State[]{(0,1)}{C}
\SmallState\State[]{(4,1)}{D}
\Initial{A}
\Final{D}
\EdgeL{A}{B}{(a,1)}
\EdgeL{C}{D}{(a,1)}
\EdgeR{A}{C}{(\star,\star)}
\EdgeL{B}{D}{(\star,\star)}
\CLoopN{A}{\{(a,0)\mid a\in \tilde{A}_\beta\}}
\CLoopN[.75]{B}{\{(a,0)\mid a\in \tilde{A}_\beta\}}
\CLoopS{C}{\{(a,0)\mid a\in \tilde{A}_\beta\}}
\CLoopS[.75]{D}{\{(a,0)\mid a\in \tilde{A}_\beta\}}
\end{VCPicture}
\caption{Automaton for $X_a$.}
\label{figure: X_a}
\end{figure}
\end{proof}

We are now ready to prove Theorem~\ref{thm: definissable implique reconnaissable}.

\begin{proof}[Proof of Theorem~\ref{thm: definissable implique reconnaissable}]
We follow the lines of~\cite{BHMV92} by proving the result 
by induction on the length of the formula. If $\phi(x_1,\dots,x_n)$ is a formula written in the structure 
$\langle \mathbb{R}, 1, \le, +,  X_\beta\rangle$, 
we let $X_\phi$ denote the set $\{(x_1,\dots,x_n) \in \mathbb{R}^n \mid \phi(x_1,\dots,x_n) \text{ is true}\}$ 
and we let $\mathcal{A}_\phi$ denote a B\"uchi automaton (when it exists) accepting $d_\beta(X_\phi)$.

The existence of $\mathcal{A}_\phi$ for atomic formulae $\phi$ 
comes from Lemma~\ref{lemma: R^n recognizable}, Lemma~\ref{lemma: Z^n rec}, 
Corollary~\ref{corollary: + rec} and Lemma~\ref{lemma: X_beta rec}.
We have to prove that $\beta$-recognizability is preserved under disjunction, negation and existence, i.e., 
we prove that if $X_\phi$ and $X_\psi$ are $\beta$-recognizable, 
then so are $X_{\phi \vee \psi}$, $X_{\neg \phi}$ and $X_{\exists x \phi}$. 
As a consequence, we obtain that $X_{\phi \wedge \psi} = X_{\neg (\neg \phi \vee \neg \psi)}$, 
$X_{\phi \implies\quad \psi} = X_{(\neg \phi) \vee (\phi \wedge \psi)}$ 
and $X_{\forall x \phi} = X_{\neg(\exists x \neg \phi)}$ are also $\beta$-recognizable. 
This is a consequence of Fact~\ref{fact: buchi stable} and Lemma~\ref{lemma: R^n recognizable}:
\begin{enumerate}
	\item Given a formula $\varphi(x_1,\dots,x_n,y_1,\dots,y_m,z_1,\dots,z_\ell)$ defined by 
	\[
	\phi(x_1,\dots,x_n,y_1,\dots,y_m) \vee \psi(y_1,\dots,y_m,z_1,\dots,z_\ell).
	\]
	The language $d_\beta(X_{\varphi})$ is equal to 
	$(d_\beta(X_{\phi}) \times d_\beta(\mathbb{R}^\ell)) \cap (d_\beta(\mathbb{R}^n) \times d_\beta(X_\psi))$, 
	which is $\omega$-regular if so are $d_\beta(X_{\phi})$ and $d_\beta(X_{\psi})$. 

	\item Given a formula $\phi(x_1,\dots,x_n)$, the language $d_\beta(X_{\neg \phi})$ 
	is equal to $d_\beta(\mathbb{R}^n) \setminus d_\beta(X_{\phi})$, 
	which is $\omega$-regular if so is $d_\beta(X_{\phi})$.

	\item Given a formula $\phi(x,x_1,\dots,x_n)$, the language $d_\beta(X_{\exists x \phi})$ 
	is equal to $(\pi_{[2,n+1]}(d_\beta(X_{\phi}))$, 
	where $\pi_{[2,n+1]}$ is the projection over the last $n$ components, 
	and $(\pi_{[2,n+1]}(d_\beta(X_{\phi}))$ is $\omega$-regular if so is $d_\beta(X_{\phi})$.
\end{enumerate}
\end{proof}

\section{Towards a Cobham-like theorem for $\beta$-numeration systems}
\label{section: gdifs}

The famous theorem of Cobham from 1969 states that recognizability in integer bases of sets 
of integers strongly depends on the chosen base $b$ \cite{Cobham:1969}. 
If $b$ and $b'$ are such that $\log(b)/\log(b')$ is irrational, 
then the only sets $X \subset \mathbb{N}$ that are simultaneously $b$- and $b'$-recognizable 
are the finite unions of arithmetic progressions. 
In~\cite{Boigelot&Brusten&Bruyere:2008,Boigelot&Brusten&Leroux:2009}, 
the authors studied particular $\beta$-recognizable sets of real numbers 
and obtained a Cobham-like theorem about these sets. 
The authors of~\cite{Adamczewski&Bell:2011} independently obtained almost the same result with other techniques. 
By means of graph directed iterated function systems, we provide here a translation between the two results. 
In doing so we also answer a conjecture of~\cite{Adamczewski&Bell:2011} and improve a result of~\cite{Feng&Wang:2009}. 
Prior to this, we prove that when $\beta$ is a Pisot number, the sets $X \subset \mathbb{R}^n$ 
that are $\beta$-recognizable or $\beta^k$-recognizable are the same.

\subsection{Taking powers of the base does not change recognizability}

\begin{prop}
Let $\beta$ be a Pisot number. For all positive integers $k$, 
a set $X \subset \mathbb{R}^n$ is $\beta$-recognizable if and only if it is $\beta^k$-recognizable.  
\end{prop}

\begin{proof}
Let $k$ be an integer greater than 1 (the case $k=1$ is obvious). 
The number $\beta$ being a Pisot number, $\beta^k$ is also Pisot~\cite{Pisot:1946}. 
We define the morphism $\zeta_k$ by $\zeta_k(i) = 0^{k-1} i$ for all $i \in  \tilde{A}_{\beta^k}$ and $\zeta_k(\star)=\star$. 
To alleviate the proof, we only prove the result for $n=1$. 
The general case can be handled by considering multidimensional versions of the transducers we use.

If $X \subset \mathbb{R}$ is $\beta^k$-recognizable, 
the language $d_{\beta^k}(X)$ is $\omega$-regular 
so the language $\zeta_k\left(d_{\beta^k}(X)\right)$ is also $\omega$-regular, see Fact~\ref{fact: buchi stable}. 
Moreover, any element $u$ in $\zeta_k\left(d_{\beta^k}(X)\right)$ is a $\beta$-representation of an element of $X$: 
there exists $x \in X$ such that $\val_\beta(u)=x$. 
More precisely, if $x\in d_{\beta^k}(X)$, then $\val_\beta(\zeta_k(x))=\val_{\beta^k}(x)$.
If $\mathcal{T}_{\beta,\tilde{A}_{\beta^k}}$ is the transducer of Proposition~\ref{proposition: normalization}, 
the language $\mathcal{T}_{\beta,\tilde{A}_{\beta^k}} \left( \zeta_k\left(d_{\beta^k}(X)\right) \right)$ is $\omega$-regular 
so $X$ is $\beta$-recognizable.

Now suppose that $X \subset \mathbb{R}$ is $\beta$-recognizable 
and let $\mathcal{A}_X$ denote the B\"uchi automaton accepting $d_{\beta}(X)$. 
Let $R_\beta \subset (\tilde{A}_{\beta^k}^+ \star \tilde{A}_{\beta^k}^\omega) 
\times (\tilde{A}_{\beta}^+ \star \tilde{A}_{\beta}^\omega)$ 
be the relation defined by $\mathcal{T}_{\beta,\tilde{A}_{\beta^k}}$.
It is $\omega$-regular by Lemma~\ref{fact: transducer conserve la regularite}.
Let $L_{\zeta_k}$ be the $\omega$-regular language 
$(0^{k-1} \tilde{A}_{\beta^k})^+ \star (0^{k-1} \tilde{A}_{\beta^k})^\omega$ 
accepted by the B\"uchi automaton $\mathcal{A}_{\zeta_k}$ depicted in Figure~\ref{figure: zeta^k}. 
Since $d_{\beta}(X)$ is $\omega$-regular, if $\pi_1$ is the projection on the first component, 
the language $L_X := \pi_1\left(R_{\beta} \cap (L_{\zeta_k} \times 0^* d_{\beta}(X))\right) 
\subset \tilde{A}_{\beta^k}^+\star \tilde{A}_{\beta^k}^\omega$ is $\omega$-regular. 
Furthermore, due to the intersection with $R_{\beta}$, any word $u$ in $L_X$ is such that $\val_\beta(u) \in X$ 
and, conversely, for all $x \in X$, the word $\zeta_k(d_{\beta^k}(x))$ belongs to 
$(0^{k-1} \tilde{A}_{\beta^k})^+ \star (0^{k-1} \tilde{A}_{\beta^k})^\omega$ 
and is such that $\left(\zeta_k(d_{\beta^k}(x)),0^{m(x)}d_{\beta}(x)\right)$ belongs to $R_\beta$ for some integer $m(x)$.
Thus we have $\val_{\beta}(L_X) = X$. 

Let $M_X:=\zeta_k^{-1}(L_X)$. 
This language is $\omega$-regular by \cite[Proposition 5.5]{PerrinPin04} (also see Fact~\ref{fact: buchi stable})
and $\val_{\beta^k}(M_X) = X$. 
Thanks to the transducer $\mathcal{T}_{\beta^k,\tilde{A}_{\beta^k}}$, the language $d_{\beta^k}(X)$ is $\omega$-regular.

\begin{figure}[h!tbp]
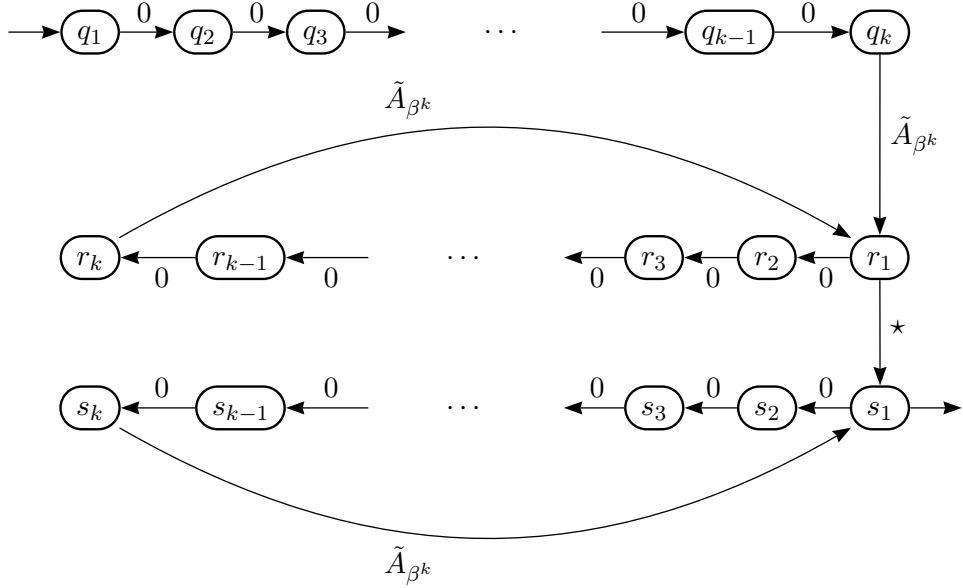

\centering
\begin{VCPicture}{(0,0)(10.5,8)}
\ChgStateLineWidth{.5}
\ChgEdgeLineWidth{.5}
\ChgEdgeLabelScale{.6}
\ChgStateLabelScale{.6}
\SmallState\StateVar[q_1]{(0,7)}{1}
\SmallState\StateVar[q_2]{(1.5,7)}{2}
\SmallState\StateVar[q_3]{(3,7)}{3}
\SmallState\StateVar[q_{k-1}]{(8.5,7)}{6}
\SmallState\StateVar[q_k]{(10.5,7)}{7}

\SmallState\StateVar[r_1]{(10.5,4)}{111}
\SmallState\StateVar[r_2]{(9,4)}{222}
\SmallState\StateVar[r_3]{(7.5,4)}{333}
\SmallState\StateVar[r_{k-1}]{(2,4)}{666}
\SmallState\StateVar[r_k]{(0,4)}{777}

\SmallState\StateVar[s_1]{(10.5,2)}{11}
\SmallState\StateVar[s_2]{(9,2)}{22}
\SmallState\StateVar[s_3]{(7.5,2)}{33}
\SmallState\StateVar[s_{k-1}]{(2,2)}{66}
\SmallState\StateVar[s_k]{(0,2)}{77}

\ChgStateLineColor{white}
\SmallState\StateVar[]{(4.5,7)}{4}
\SmallState\StateVar[]{(6.5,7)}{5}
\SmallState\StateVar[\cdots]{(5.5,7)}{4.5}

\SmallState\StateVar[]{(6,4)}{444}
\SmallState\StateVar[]{(4,4)}{555}
\SmallState\StateVar[\cdots]{(5,4)}{4.555}

\SmallState\StateVar[]{(6,2)}{44}
\SmallState\StateVar[]{(4,2)}{55}
\SmallState\StateVar[\cdots]{(5,2)}{4.55}

\Initial{1}
\Final{11}

\EdgeL{1}{2}{0}
\EdgeL{2}{3}{0}
\EdgeL{3}{4}{0}
\EdgeL{5}{6}{0}
\EdgeL{6}{7}{0}

\EdgeL{111}{222}{0}
\EdgeL{222}{333}{0}
\EdgeL{333}{444}{0}
\EdgeL{555}{666}{0}
\EdgeL{666}{777}{0}
\LArcL{777}{111}{\tilde{A}_{\beta^k}}

\EdgeR{11}{22}{0}
\EdgeR{22}{33}{0}
\EdgeR{33}{44}{0}
\EdgeR{55}{66}{0}
\EdgeR{66}{77}{0}
\LArcR{77}{11}{\tilde{A}_{\beta^k}}

\EdgeL{7}{111}{\tilde{A}_{\beta^k}}
\EdgeL{111}{11}{\star}

\end{VCPicture}
\caption{Automaton for $L_{\zeta_k}$}
\label{figure: zeta^k}
\end{figure}
\end{proof}

\subsection{Background on Cobham theorem for real numbers in integer bases}

\begin{definition}
Two real numbers $\beta$ and $\gamma$ greater than $1$ are {\em multiplicatively independent} 
if $\frac{\log \beta}{\log \gamma} \notin \mathbb{Q}$.
\end{definition}

\begin{definition}
Let $b \geq 2$ be an integer. A compact set $X \subset [0,1]^n$ is {\em $b$-self-similar} 
if its $b$-kernel is finite where the {\em $b$-kernel of X} is the collection of sets 
\[
	\left\{ (b^k X - \mathbf{a}) \cap [0,1]^n \mid  k \geq 0, \mathbf{a} = (a_1,\dots,a_n) \in \mathbb{Z}^n \text{ and  } (\forall i)\, 0 
	\leq a_i < b^k  \right\}.
\]
\end{definition}

Examples of $b$-self-similar sets are the Pascal's triangle modulo 2 (it is 2-self-similar) that consists in the adherence of the set
\[
	\left\{ \left( \val_2(0\star \rep_2(n)),\val_2(0\star \rep_2(m)) \right)	\mid 
	\left( 
	\begin{array}{c}
	m	\\	n
	\end{array}
	\right)	\equiv 1  \mod 2	\right\}
\]
depicted in Figure~\ref{figure:pascal} and the Menger sponge (it is 3-self-similar) 
that consists in the adherence of the set of points $\mathbf{x} \in [0,1]^3$ 
such that $\rep_3(\mathbf{x})$ does not contain occurrences\footnote{Each face 
of the Menger sponge is a Sierpinski carpet.} of digits in $\{(0,1,1),(1,0,1),(1,1,0),(1,1,1)\}$; 
it is depicted in Figure~\ref{figure:sponge}.

\begin{figure}[h!tbp]
\centering
\includegraphics[scale=0.5]{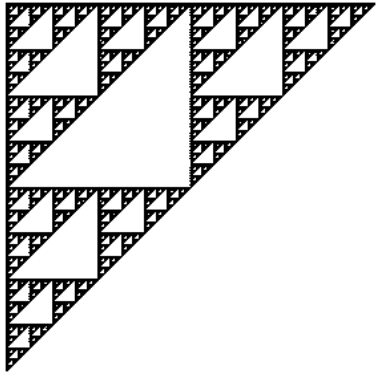}
\caption{Pascal's triangle modulo 2 is 2-self-similar.}
\label{figure:pascal}
\end{figure}

\begin{figure}[h!tbp]
\centering
\includegraphics[scale=0.5]{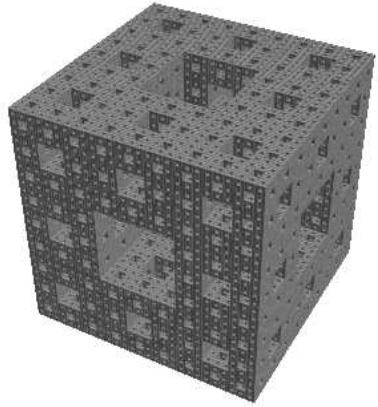}
\caption{Menger sponge is 3-self-similar.}
\label{figure:sponge}
\end{figure}

\begin{theorem}[Adamczewski and Bell~\cite{Adamczewski&Bell:2011}]
Let $b$ and $b'$ be two multiplicatively independent integers greater than 1. 
A compact set $X \subset [0,1]$ is simultaneously $b$- and $b'$-self-similar 
if and only if it is a finite union of closed intervals with rational endpoints.
\end{theorem}

\begin{conjecture}[Adamczewski and Bell~\cite{Adamczewski&Bell:2011}]
\label{conjecture:AB}
Let $b$ and $b'$ be two multiplicatively independent integers greater than 1. 
A compact set $X \subset [0,1]^n$ is simultaneously $b$- and $b'$-self-similar 
if and only if it is a finite union of polyhedra whose vertices have rational coordinates.
\end{conjecture}

The aim of this section is to prove that this conjecture holds true: 
this is a consequence of Theorem~\ref{thm cobham version weak buchi}, 
Theorem~\ref{thm:digraph automate}, Theorem~\ref{thm: digraph implique self-similar} 
and Theorem~\ref{thm:self-similar implique digraph} below.

\begin{definition}
A B\"uchi automaton $\mathcal{A} = (Q,A,E,I,T)$ is said to be {\em weak} 
if all its strongly connected components are subsets either of $T$, or of $Q \setminus T$. 
A set $X$ such that $d_\beta(X)$ is accepted by a weak automaton is said to by {\em weakly $\beta$-recognizable}.
\end{definition}

The use of weak automata finds its motivation in~\cite{Boigelot&Jodogne&Wolper:2001} 
where the authors prove that any set definable by a first order formula 
in the structure $\langle \mathbb{R}, \mathbb{Z}, +, \leq \rangle$ 
is weakly $b$-recognizable for any integer base $b \geq 2$. 
Moreover, the behaviour of such automata are comparable with finite automata: 
minimal weak Büchi automata can be defined, the class of weak Büchi automata is closed under complementation, etc.
The next result shows that the converse also holds true.

\begin{theorem}[Boigelot, Brusten, Bruyère, Jodogne, Leroux 
and Wolper~\cite{Boigelot&Jodogne&Wolper:2001,Boigelot&Brusten&Bruyere:2008,Boigelot&Brusten&Leroux:2009}]
\label{thm cobham version weak buchi}
Let $b$ and $b'$ be two multiplicatively independent integers greater than 1. 
A set $X \subset \mathbb{R}^n$ is simultaneously weakly $b$- and $b'$-recognizable 
if and only if it is definable by a first order formula in the structure $\langle \mathbb{R}, \mathbb{Z}, +, \leq \rangle$. 
\end{theorem}

\subsection{Background on graph directed iterated function systems}

By means of graph directed iterated function systems, we are able to determine 
the link between the notions of weakly $b$-recognizable and $b$-self-similar. 
As in the first part of the paper, we extend these notions to $\beta$-numeration systems for real numbers $\beta>1$.

\begin{definition}
Let $G = (V,E)$ be a directed graph. For all $u,v \in V$, we let $E_{uv}$ denote the set of edges from $u$ to $v$. 
An {\em iterated function system realizing $G$} is given by a collection of metric spaces $(X_v,\rho_v)$, $v \in V$, 
and of similarities\footnote{Recall that a {\em similarity} of ratio $r$ is a function $f:X_v \to X_u$ 
such that $\rho_u(f(\mathbf{x}), f(\mathbf{y})) = r\rho_v(\mathbf{x},\mathbf{y})$ for all $\mathbf{x},\mathbf{y} \in X_v$.} 
$S_e : X_v \to X_u$, $e \in E_{uv}$, of ratio $r_e$. 
An {\em attractor} (or {\em  invariant list}) for such an iterated function system 
is a list of nonempty compact sets $K_u \subset X_u$ such that for all $u \in V$,
\[
	K_u = \bigcup_{v \in V} \bigcup_{e \in E_{uv}} S_e(K_v).
\]
A {\em graph-directed iterated function system} ({\em GDIFS} for short) is given by a 4-tuple 
\[
	(V,E,(X_v,\rho_v)_{v \in V},(S_e)_{e \in E})
\]
with $(V,E)$ the underlying directed graph, $(X_v,\rho_v)_{v \in V}$ 
the collection of metric spaces and $(S_e)_{e \in E}$ the collection of similarities.  
\end{definition}

\begin{remark}
\label{rem: infinite path}
In the directed graph of a GDIFS, it is always assumed that each vertex $v$ has at least one edge starting from $v$. 
\end{remark}

\begin{example}[Rauzy fractal]
\label{ex:Rauzy fractal}
A classical example of attractor of GDIFS is the {\em Rauzy fractal} \cite{Rauzy:1982} 
represented in Figure~\ref{figure: Rauzy fractal}. It can obtained as follows. 
Let $\sigma$ be the tribonacci substitution defined over the alphabet 
$A=\{1,2,3\}$ by $\sigma(1) = 12$, $\sigma(2)=13$ and $\sigma(3)=1$. 
We define the abelianization map $P:A^+ \to \mathbb{R}^3$ by, 
for $u_0, \dots, u_n \in A$, $P(u_0 \cdots u_n) = \sum_{i=0}^n \mathbf{e}_{u_i}$ 
where $\{\mathbf{e}_1,\mathbf{e}_2,\mathbf{e}_3\}$ is the canonical basis of $\mathbb{R}^3$. 
The incidence matrix $M_\sigma$ of $\sigma$ is primitive and its dominating eigenvalue $\beta$ 
is a unit Pisot number (real root of $x^3-x^2-x-1$). 
Thus, the space $\mathbb{R}^3$ can be decomposed into the {\em expanding line} $\mathbb{H}_e$ 
(eigenspace of $M_\sigma$ associated with $\beta$) and the {\em contracting plane} $\mathbb{H}_c$ 
(eigenspace associated with the conjugates of $\beta$). 
Then, if $\sigma^\omega(1) = u_0 u_1 u_2 \cdots$, the sequence $(P(u_0 u_1 \cdots u_n))_{n \in \mathbb{N}}$ 
describes a broken line whose distance to $\mathbb{H}_e$ is bounded. 
If $\pi: \mathbb{R}^3 \to \mathbb{H}_c$ is the projection along $\mathbb{H}_e$, 
then the Rauzy fractal is defined by 
\[
	\mathcal{R} = \overline{\left\{ \pi\left( P(u_0 u_1 \cdots u_n) \right) \mid n \in \mathbb{N} \right\}}
\]
and its subtiles are defined by
\[
	\mathcal{T}(i) = \overline{\left\{ \pi\left( P(u_0 u_1 \cdots u_n) \right) \mid n \in \mathbb{N}, u_n = i \right\}}.
\]

\begin{figure}[h!tbp]
\centering
\includegraphics[scale=0.2]{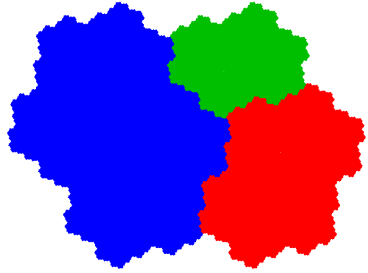}
\caption{The Rauzy fractal is the attractor of a GDIFS.}
\label{figure: Rauzy fractal}
\end{figure}

\noindent Then, the family $\{\mathcal{T}(1),\mathcal{T}(2),\mathcal{T}(3)\}$ 
satisfies the equations \cite{Arnoux&Ito:2001,Sirvent&Wang:2002}
\begin{eqnarray*}
\mathcal{T}(1) 	&=&	h(\mathcal{T}(1)) \cup h(\mathcal{T}(2)) \cup h(\mathcal{T}(3))	\\
\mathcal{T}(2) 	&=&	h(\mathcal{T}(1)) + \pi(P(1))	\\
\mathcal{T}(3) 	&=&	h(\mathcal{T}(2)) + \pi(P(1)) 
\end{eqnarray*} 
where $h$ is the similarity $\pi M_{\sigma}$.
So $(\mathcal{T}(1),\mathcal{T}(2),\mathcal{T}(3))$ is an attractor of the GDIFS 
represented in Figure~\ref{figure: GDIFS for Rauzy fractal}.

\begin{figure}[h!tbp]
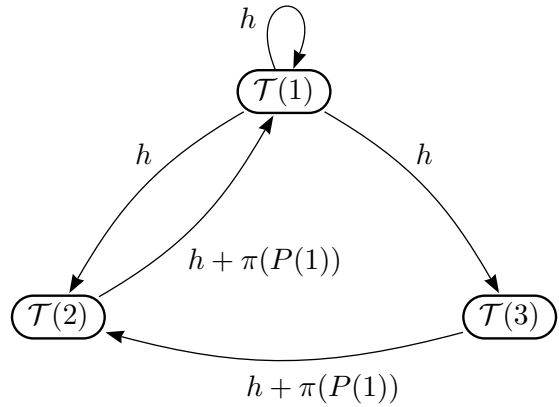

\centering
\begin{VCPicture}{(0,-1)(6,4)}
\ChgStateLineWidth{.5}
\ChgEdgeLineWidth{.5}
\ChgEdgeLabelScale{.6}
\ChgStateLabelScale{.6}
\SmallState\StateVar[\mathcal{T}(1)]{(3,3)}{1}
\SmallState\StateVar[\mathcal{T}(2)]{(0,0)}{2}
\SmallState\StateVar[\mathcal{T}(3)]{(6,0)}{3}
\ArcR{1}{2}{h}
\ArcL{1}{3}{h}
\CLoopN{1}{h}
\ArcR{2}{1}{h+\pi(P(1))}
\ArcL{3}{2}{h+\pi(P(1))}
\end{VCPicture}
\caption{GDIFS for the Rauzy fractal.}
\label{figure: GDIFS for Rauzy fractal}
\end{figure}
\end{example}

\begin{example}
The Pascal triangle modulo 2 $P$ (Figure~\ref{figure:pascal}) can be easily obtained by a GDIFS construction: we have
\[
	P = \frac{P}{2} \cup \frac{P+(0,1)}{2} \cup \frac{P+(1,1)}{2}
\]
so P is an attractor of the GDIFS represented in Figure~\ref{figure:digraph pascal}.
\begin{figure}[h!tbp]
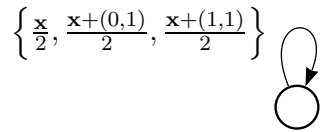

\centering
\begin{VCPicture}{(0,-0.3)(2,1.5)}
\ChgStateLineWidth{.5}
\ChgEdgeLineWidth{.5}
\ChgEdgeLabelScale{.6}

\SmallState\State[]{(1,0)}{A}
\CLoopN{A}{\left\{\frac{\mathbf{x}}{2},\frac{\mathbf{x}+(0,1)}{2},\frac{\mathbf{x}+(1,1)}{2} \right\}}
\end{VCPicture}
\caption{GDIFS associated with the Pascal triangle.}
\label{figure:digraph pascal}
\end{figure}

Similarly, the Menger sponge $MS$ (Figure~\ref{figure:sponge}) can be obtained by the following GDIFS construction: 
if  $E = \{0,1,2\}^3 \setminus \{(0,1,1),(1,0,1),(1,1,0),(1,1,1)\}$, we have
\[
	MS = \bigcup_{\mathbf{e} \in E} \frac{MS+\mathbf{e}}{3}
\] 
so MS is an attractor of the GDIFS represented in Figure~\ref{figure:digraph sponge}.
\begin{figure}[h!tbp]
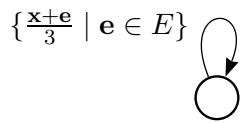

\centering
\begin{VCPicture}{(0,-0.3)(2,1.5)}
\ChgStateLineWidth{.5}
\ChgEdgeLineWidth{.5}
\ChgEdgeLabelScale{.6}

\SmallState\State[]{(1,0)}{A}
\CLoopN{A}{\{\frac{\mathbf{x}+\mathbf{e}}{3} \mid \mathbf{e} \in E\}}
\end{VCPicture}
\caption{GDIFS associated with the Menger sponge.}
\label{figure:digraph sponge}
\end{figure}
\end{example}

The previous  two examples are actually obtained by an {\em iterated function system} ({\em IFS} for short): 
the underlying directed graph is useless. 
Furthermore, they are obtained by a so called {\em homogeneous} IFS 
because all similarities are affine maps with the same contraction ratio. 
Finally, these homogeneous IFS satisfy the {\em open set condition}: 
there exists an open set $V$ such that $S_i(V) \cap S_j(V)=\emptyset$ for any two distinct similarities $S_i$ and $S_j$. 
In that setting, Feng and Wang proved, among other things, the following result which is closely linked to our work.

\begin{theorem}[Feng and Wang~\cite{Feng&Wang:2009}]
\label{thm: IFS}
Let $\Phi = \{\phi_i\}_{i=1}^N$ and $\Psi = \{\psi_j\}_{j=1}^M$ be two homogeneous IFS 
with metric spaces $\mathbb R$ embedded with the Euclidean distance, 
with contraction ratios $r_\Phi$ and $r_\Psi$ respectively,
and that satisfy the open set condition. 
Suppose that $X \subset \mathbb{R}$ is an attractor of both $\Phi$ and $\Psi$.
\begin{enumerate}
	\item	If \footnote{$\dim_H$ stands for Hausdorff dimension.} $\dim_H(X)=s<1$, 
			then $\frac{\log|r_\Phi|}{\log|r_\Psi|} \in \mathbb{Q}$;
	\item	If $\dim_H(X)=1$ and $X$ is not a finite union of intervals, 
			then $\frac{\log|r_\Phi|}{\log|r_\Psi|} \in \mathbb{Q}$.
\end{enumerate}
\end{theorem}

The result we obtain with our translation is therefore stronger than Feng and Wang's one in two directions: 
it is not limited to dimension 1 and it concerns GDIFS instead of IFS. 
It is also weaker in the sense that the similarities are affine functions 
with contraction ratios of the form $1/b$ for $b \in \mathbb{Z}_\geq 2$.

Let us provide an example of GDIFS with a non-trivial directed graph.

\begin{example}
\label{ex: cantor}
Let $\mathcal{G}$ be the GDIFS whose directed graph is represented in Figure~\ref{figure: ex digraph cantor}, 
whose metric spaces are $\mathbb{R}$ with the Euclidean distance and whose similarities are $S_i: 
x \mapsto \frac{x+i}{3}$ for $i \in \{-2,0,2\}$.
If $X_T$ is the usual triadic Cantor set that consist of all real numbers in $[0,1]$ 
whose ternary expansion does not contain the digit $1$, then $\mathcal{G}$ admits the attractor 
$(X_T,-X_T,X_T \cup (-X_T))$. 
Indeed, $X_T$ corresponds to the vertex $C$, $-X_T$ to the vertex $B$ and $X_T \cup (-X_T)$ to the vertex $A$ since
\begin{eqnarray*}
	X_T 					&=& 		S_0(X_T) 	\cup 	S_2(X_T)			\\
	-X_T 				&=& 		S_0(-X_T) 	\cup 	S_{-2}(-X_T)		\\
	X_T \cup (-X_T) 		&=& 		S_0(X_T) 	\cup 	S_2(X_T)		\cup S_0(-X_T) 	\cup 	S_{-2}(-X_T)	\\
\end{eqnarray*}

\begin{figure}[h!tbp]
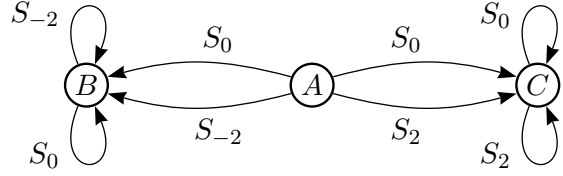

\centering
\begin{VCPicture}{(0,0)(6,2)}
\ChgStateLineWidth{.5}
\ChgStateLabelScale{.6}
\ChgEdgeLineWidth{.5}
\ChgEdgeLabelScale{.6}

\SmallState\State[A]{(3,1)}{A}
\SmallState\State[B]{(0,1)}{B}
\SmallState\State[C]{(6,1)}{D}
\ArcR{A}{B}{S_0}
\ArcL{A}{D}{S_0}
\ArcL{A}{B}{S_{-2}}
\ArcR{A}{D}{S_2}
\CLoopS{B}{S_0}
\CLoopN{B}{S_{-2}}
\CLoopN{D}{S_0}
\CLoopS{D}{S_2}
\end{VCPicture}
\caption{Underlying directed graph of the GDIFS representing $X_T \cup (-X_T)$.}
\label{figure: ex digraph cantor}
\end{figure}
\end{example}

The following result is well known.

\begin{theorem}[\cite{Edgar08}] \label{thm: digraph}
If $\mathcal{G} = (V,E,(X_v,\rho_v)_{v \in V},(S_e)_{e \in E})$ is a GDIFS 
such that all $(X_v,\rho_v)$ are nonempty complete metric spaces and all similarities $S_e$ have a ratio $r_e < 1$, 
then $\mathcal{G}$ admits a unique attractor.
\end{theorem}

In what follows, the complete metric spaces $(X_v,\rho_v)$ are always $\mathbb{R}^n$ with the Euclidean distance.

\subsection{Link between $\beta$-recognizable sets and sets obtained by some GDIFS construction}

From the GDIFS point of view, we do not need to take care about properties of $\beta$ 
such as being Parry or Pisot, neither about limiting the coefficients 
of the similarities to canonical alphabets such as $A_\beta$. 
Thus, we consider another notion of recognizability as follows.

\begin{definition}
Let $\beta$ be a real number greater than 1 and let $C$ be the alphabet 
$\{-c,-c+1,\dots,0,\dots,c\}$ for some $c \in \mathbb{Z}_{\geq 1}$. 
As in Section~\ref{subsection: recognizable and definable}, we defined the alphabet
\[
	{\bf C}=\underbrace{C \times \cdots \times C}_{n\text{ times}}.
\] 
A set $X \subset \mathbb{R}^n$ is {\em $(\beta,C)$-acceptable} 
if there exists a B\"uchi automaton $\mathcal{A}$ over $\mathbf{C}$ 
such that $\val_\beta(\mathbf{0} \star L(\mathcal{A})) = X$.
We say that $\mathcal{A}$ is an underlying automaton for $X$.
\end{definition}

\begin{remark}
Let $C=\{-c,\ldots,c\}$ for some $c\in \mathbb{Z}_{\ge 1}$.
A $(\beta,C)$-acceptable set is always a subset of $\left[\frac{-c}{\beta-1},\frac{c}{\beta-1} \right]^n$. 
\end{remark}

\begin{remark}
Thanks to Proposition~\ref{proposition: normalization}, Fact~\ref{fact: 3 ways to recognize} 
and Lemma~\ref{fact: transducer conserve la regularite}, if $\beta$ is a Pisot number, 
then for all alphabets $C \subset \mathbb{Z}$, any $(\beta,C)$-acceptable set is $\beta$-recognizable.
\end{remark}

\begin{definition}
Let $\mathcal{A}=( Q,A,E,I,T )$ be an automaton. A state $q \in Q$ is said to be {\em accessible} 
(resp. {\em co-accessible}) if there is a path from $I$ to $q$ (resp. from $q$ to $T$). 
An automaton is said to be {\em trim} if all its states are accessible and co-accessible. 
Given an automaton $\mathcal{A}$, we can also build a trim automaton $\mathcal{A}'$ 
that accepts the same language (just by removing non accessible and non co-accessible states). 
We say that an automaton is {\em closed} if for all states $q$ such that there is a nonempty path from $q$ to $q$, $q$ is final.
\end{definition}

\begin{remark}
If an accessible B\"uchi automaton $\mathcal{A}$ is closed then, for all infinite words $w$ labeling a path in $\mathcal{A}$ not necessarily starting from an initial state, 
$w$ is an accepting tail, i.e., there is an infinite word $v$ accepted by $\mathcal{A}$ that admits $w$ as a suffix.
\end{remark}

\begin{remark}\label{rem:closedweak}
Any closed trim B\"uchi automaton is weak: all its strongly connected components are subsets of $T$.
\end{remark}

\begin{theorem} \label{thm:digraph automate}
Let $\beta$ be a real number greater than 1 and let $C$ be the alphabet 
$\{-c,-c+1,\dots,0,\dots,c\}$ for $c \in \mathbb{Z}_{\geq 1}$.
A set $X \subset \mathbb{R}^n$ is $(\beta,C)$-acceptable set whose underlying 
trim B\"uchi automaton is closed if and only if $X$ is a finite union of compact sets 
belonging to the attractor of a GDIFS whose similarities are of the form 
$S_\mathbf{a}(\mathbf{x}) = \frac{\mathbf{x}+\mathbf{a}}{\beta}$ for $\mathbf{a} \in \mathbf{C}$
\end{theorem}

We first need the following lemma.

\begin{lemma} \label{lemma:trim closed}
If the underlying trim B\"uchi automaton of a $(\beta,C)$-acceptable set $X$ is closed, then $X$ is a (topologically) closed set.
\end{lemma}

\begin{proof}
Consider a sequence $(\mathbf{x}_n)_{n \in \mathbb{N}}$ of elements of $X$ 
that converges to some $\mathbf{x} \in \mathbb{R}^n$ and let us prove that $\mathbf{x}$ belongs to $X$. 
Let $(\mathbf{w}_n)_{n \in \mathbb{N}}$ be a sequence of infinite words 
over $\mathbf{C}$ accepted by $\mathcal{A}$ such that 
$\val_\beta(\mathbf{0} \star \mathbf{w}_n) = \mathbf{x}_n$ for all $n$. 
There is no reason for which $(\mathbf{w}_n)_{n \in \mathbb{N}}$ would converge to a limit word. 
However, by the pigeonhole principle, $(\mathbf{w}_n)_{n \in \mathbb{N}}$ admits a subsequence 
$(\mathbf{w}_{k(n)})_{n \in \mathbb{N}}$ such that $\mathbf{w}_{k(n)}$ and $\mathbf{w}_{k(n+1)}$ 
share a prefix of length at least $n$. 
Thus, $(\mathbf{w}_{k(n)})_{n \in \mathbb{N}}$ converges to some word $\mathbf{w}$ 
over $\mathbf{C}$ that labels an infinite path in $\mathcal{A}$ starting from an initial state
and such that $\val_\beta(\mathbf{0} \star \mathbf{w})=\mathbf{x}$. 
Since $\mathcal{A}$ is closed, $\mathbf{w}$ is accepted by $\mathcal{A}$ so we have $\mathbf{x} \in X$.
\end{proof}

\begin{remark}
If a set $X$ is $(\beta,C)$-acceptable with underlying trim automaton $\mathcal{A}$, then its adherence is also $(\beta,C)$-acceptable and an underlying automaton for it is $\mathcal{A}$ where all states are final; this automaton is closed. 
Thus, for $\beta$ Pisot, any $(\beta,C)$-acceptable set $X \subset [-1,1]^n$ is $\beta$-recognizable and its adherence belongs to the small rectangle in Figure~\ref{fig: general picture} (page~\pageref{fig: general picture}).
\end{remark}

\begin{proof}[Proof of Theorem~\ref{thm:digraph automate}]
Let $\mathcal{A} = ( Q,\mathbf{C},E,I,T )$ be a trim B\"uchi automaton over $\mathbf{C}$ 
which is closed and such that $\val_\beta(\mathbf{0} \star L(\mathcal{A})) = X$. 
The GDIFS that we build is obtained from $\mathcal{A}$ by considering, for all states $q \in Q$, 
the complete metric space $X_q = \mathbb{R}^n$ (with the Euclidean distance) 
and by replacing the label $\mathbf{a} \in \mathbf{C}$ or each transition by the similarity 
$S_\mathbf{a}\colon \mathbf{x} \mapsto \frac{\mathbf{x}+\mathbf{a}}{\beta}$ of ratio $1/\beta <1$. 
By Theorem~\ref{thm: digraph}, the GDIFS admits a unique attractor $(K_q,\, q \in Q)$: for all $q \in Q$,
\[
	K_q = \bigcup_{p \in Q} \bigcup_{q \xrightarrow[]{\mathbf{a}} p} S_\mathbf{a}(K_p).
\]
Let us prove that 
\[
	X = \bigcup_{q \in I} K_q.
\]
For all states $q \in Q$, we let $W_q$ denote the set of infinite words that, 
starting from $q$, label an infinite path in $\mathcal{A}$ (hence the tail of an accepted run since $ \mathcal{A}$ is closed). 
We also let $Y_q$ denote the set $\{\val_\beta(\mathbf{0} \star w)\mid w\in W_q\}$. 
The automaton $\mathcal{A}$ being closed, all sets $Y_q$ are closed by Lemma~\ref{lemma:trim closed} 
and so are compact sets of $\mathbb{R}^n$ that satisfy
\[
	X = \bigcup_{q \in I} Y_q.
\]
Now let us show that $Y_q=K_q$ for all $q\in Q$.
By uniqueness of the attractor of the GDIFS, it suffices to show that the list $(Y_q, \, q \in Q)$ satisfies
\[
	\forall q \in Q, \quad Y_q = \bigcup_{p \in Q} \bigcup_{q \xrightarrow[]{\mathbf{a}} p} S_\mathbf{a}(Y_p).
\]
This is clear from the following two observations:
\begin{eqnarray*}
	\forall q \in Q, 			
	& & 
	W_q = \bigcup_{p \in Q} \bigcup_{q \xrightarrow[]{\mathbf{a}} p} \mathbf{a} W_p \\
	\forall \mathbf{w} \in \mathbf{C}^\omega, \forall \mathbf{a} \in \mathbf{C}, 
	& &
	\val_\beta (\mathbf{0} \star \mathbf{a} \mathbf{w}) = S_\mathbf{a} \left( \val_\beta(\mathbf{0} \star \mathbf{w}) \right).
\end{eqnarray*}

Now let $\mathcal{G}$ be a GDIFS with underlying directed graph $G = (V,E)$, $V = \{K_1,\dots,K_m\}$, 
complete metric spaces $\mathbb{R}^n$ with Euclidean distance and similarities of the form 
$S_\mathbf{a}\colon \mathbf{x} \mapsto \frac{\mathbf{x}+\mathbf{a}}{\beta}$ for $\mathbf{a} \in \mathbf{C}$. 
Suppose that $X$ is the union of the compact sets $K_1,\dots,K_l$, $l \leq m$.

Let $\mathcal{A}$ be the closed automaton $( V, \mathbf{C}, E, \{K_1,\dots,K_l\}, V)$ 
where the transitions correspond to the edges of $G$ 
in which we have replaced the label $S_\mathbf{a}$ by $\mathbf{a}$. 
Remark~\ref{rem: infinite path} ensures that, starting from any state $q$, 
we can read infinite words so this automaton can be viewed as a B\"uchi automaton. 
Let $Y$ be the set $\val_\beta(\mathbf{0} \star L(\mathcal{A}))$. 
By Lemma~\ref{lemma:trim closed}   $Y$ is closed since so is $\mathcal{A}$. 
We also have $Y = X$ because if we consider the GDIFS built from $\mathcal{A}$ 
as we did in the first part of the proof, we would obtain the initial GDIFS whose attractor is 
$\{K_i \mid 1 \leq i \leq m\}$ and we would get $Y = \bigcup_{1 \leq i \leq l} K_i = X$.
\end{proof}

Now let us extend the definition of $b$-kernel and of $b$-self-similar set to real numbers 
$\beta$ and arbitrary alphabets $C$. Let $\beta$ be a real number greater than 1 
and let $C$ be the alphabet $\{-c,-c+1,\dots,0,\dots,c\}$ for $c \in \mathbb{Z}_{\geq 1}$.
We let $C[\beta]$ denote the set of polynomials in $\beta$ with coefficients in $C$, i.e., 
$C[\beta] = \{\val_\beta(u \star 0^\omega) \mid u \in C^+\}$ 
and for $k \in \mathbb{Z}_{\geq 0}$, we let $C^{<k}[\beta]$ 
denote the set of polynomials of degree less than $k$ in $\beta$ with coefficients in $C$, 
i.e., $C^{<k}[\beta] = \{\val_\beta(u \star 0^\omega) \mid u \in C^{\leq k}\}$. 
We extend these notation to the multidimensional case, i.e.,
\[
	\mathbf{C}[\beta] = \underbrace{C[\beta] \times \cdots \times C[\beta]}_{n \text{ times}}	
	\quad \text{and} \quad 
	\mathbf{C}^{<k}[\beta] = \underbrace{C^{<k}[\beta] \times \cdots \times C^{<k}[\beta]}_{n \text{ times}}.
\]

\begin{definition}
A compact set $X \subset \left[\frac{-c}{\beta -1},\frac{c}{\beta -1} \right]^n$ 
is {\em $(\beta,C)$-self-similar} if its $(\beta,C)$-kernel is finite, where the {\em $(\beta,C)$-kernel} 
of $X$ is the family of sets
\[
	N_{k,\mathbf{b}}(X) = \beta^k X - \mathbf{b} \cap \left[\frac{-c}{\beta -1},\frac{c}{\beta-1} \right]^n,
\]
with $k \in \mathbb{Z}_{\geq 0}$ and $\mathbf{b} \in \mathbf{C}^{<k}[\beta]$.
\end{definition}

\begin{example}
Let us consider the notation of Example~\ref{ex: cantor}.
If $\tilde{A}_3$ is the alphabet $\{-2,-1,0,1,2\}$, then the set $X_T \cup (-X_T)$ 
is $(3,\tilde{A}_3)$-self-similar since its $(3,\tilde{A}_3)$-kernel is 
$\{X_T \cup (-X_T),X_T \cup \{-1\},-X_T \cup \{1\},X_T,-X_T,\{1\},\{-1\},\emptyset\}$.
\end{example}

\begin{theorem}
\label{thm: digraph implique self-similar}
Let $\beta$ be a Pisot number and $C$ be the alphabet 
$\{-c,-c+1,\dots,0,\dots,c\}$ for $c \in \mathbb{Z}_{\geq 1}$. If $X \subset \mathbb{R}^n$ 
is a finite union of compact sets belonging to the attractor of a GDIFS whose similarities are of the form 
$S_\mathbf{a}(\mathbf{x}) = \frac{\mathbf{x}+\mathbf{a}}{\beta}$ for $\mathbf{a} \in \mathbf{C}$, 
then $X$ is $(\beta,C)$-self-similar.
\end{theorem}

We will need the following result where, again, algebraic properties of $\beta$ enter the picture.

\begin{theorem}[Berend and Frougny~\cite{Berend-Frougny94,Frougny-CANT}]
\label{lemma: frougny}
Let $\beta>1$ be a real number. The set $C[\beta] \cap \left[\frac{-c}{\beta -1},\frac{c}{\beta-1} \right]$ 
is finite for all alphabets $C = \{-c,-c+1,\dots,0,\dots,c\}$, $c \in \mathbb{Z}_{\geq 1}$, if and only if $\beta$ is a Pisot number.
\end{theorem}


\begin{proof}[Proof of Theorem~\ref{thm: digraph implique self-similar}]
By Theorem~\ref{thm: digraph}, the GDIFS has a unique attractor $\{K_i \mid 1 \leq i \leq m\}$. 
Suppose that $X = K_1 \cup K_2 \cup \dots \cup K_r$, $r \leq m$.
For all vertices $p$ and $q$ of the directed graph, 
we let $E_{pq}^\ell$ denote the set of words $\mathbf{a}_1 \cdots \mathbf{a}_\ell$ of length $\ell$ 
over $\mathbf{C}$ for which there is a sequence of vertices $q = q_0, \dots, q_\ell = p$ 
of the directed graph such that for all $i$, $S_{\mathbf{a}_i}$ labels an edge from $q_{i-1}$ to $q_i$.
For all $i \in \{1,\dots,m\}$ and all $\ell \geq 1$, it comes
\[
	K_i = \bigcup_{j=1}^m \bigcup_{\mathbf{a}_1 \cdots \mathbf{a}_\ell \in E_{ij}^\ell} 
		S_{\mathbf{a}_1} \circ \cdots \circ S_{\mathbf{a}_\ell}(K_j)
\]
and so
\begin{equation}
\label{eq: digraph}
	X = \bigcup_{i=1}^r \bigcup_{j=1}^m \bigcup_{\mathbf{a}_1 \cdots \mathbf{a}_\ell 
		\in E_{ij}^\ell} S_{\mathbf{a}_1} \circ \cdots \circ S_{\mathbf{a}_\ell}(K_j).
\end{equation}

Let us prove that the collection of sets 
\[
	N_{\ell,\mathbf{b}}(X) = \beta^\ell X - \mathbf{b} \cap \left[\frac{-c}{\beta -1},\frac{c}{\beta-1} \right]^n, \quad 
	\ell \in \mathbb{Z}_{\geq 0}, \, 
	\mathbf{b} \in \mathbf{C}^{\leq \ell}[\beta]
\]
is finite. Using~\eqref{eq: digraph}, we get, for some $\ell$ and $\mathbf{b}$,
\begin{eqnarray*}
	N_{\ell,\mathbf{b}}(X) 
	&=& 
	\left(
	\beta^\ell 
	\left( 
	\bigcup_{i=1}^r \bigcup_{j=1}^m \bigcup_{\mathbf{a}_1 \cdots \mathbf{a}_\ell 
			\in E_{ij}^\ell} S_{\mathbf{a}_1} \circ \cdots \circ S_{\mathbf{a}_\ell}(K_j) 
	\right)
	- \mathbf{b} 
	\right)
	\cap 
	\left[\frac{-c}{\beta -1},\frac{c}{\beta-1} \right]^n\\
	&=& 
	\bigcup_{i=1}^r \bigcup_{j=1}^m \bigcup_{\mathbf{a}_1 \cdots \mathbf{a}_\ell \in E_{ij}^\ell} 
	\left(
	\left(
	\beta^\ell 
	S_{\mathbf{a}_1} \circ \cdots \circ S_{\mathbf{a}_\ell}(K_j)	
	- \mathbf{b} 
	\right)
	\cap 
	\left[\frac{-c}{\beta -1},\frac{c}{\beta-1} \right]^n	
	\right)
\end{eqnarray*}
Let $D$ be the alphabet $\{-2c,-2c+1,\dots,0,\dots,2c\}$. 
Observe that we have
\[
	\beta^\ell \left( S_{\mathbf{a}_1} \circ \cdots \circ S_{\mathbf{a}_\ell}(K_j) \right)- \mathbf{b}
	=
	K_j + \left( \mathbf{a}_\ell + \beta \mathbf{a}_{\ell-1} + \beta^2 \mathbf{a}_{\ell-2} 
	+ \cdots + \beta^{\ell-1} \mathbf{a}_1\right) - \mathbf{b}
\]
and
\[
	x = \mathbf{a}_\ell + \beta \mathbf{a}_{\ell-1} + \beta^2 \mathbf{a}_{\ell-2} + \cdots 
		+ \beta^{\ell-1} \mathbf{a}_1 - \mathbf{b} \in \mathbf{D}[\beta],
\]
where
\[
	{\bf D}=\underbrace{D \times \cdots \times D}_{n\text{ times}}.
\]
Since $K_j$ is included in $\left[\frac{-c}{\beta -1},\frac{c}{\beta-1} \right]^n$, 
Theorem~\ref{lemma: frougny} implies that there is only a finite number of $\mathbf{x} \in \mathbf{D}[\beta]$ such that
\[
	(K_j + \mathbf{x}) \cap \left[\frac{-c}{\beta -1},\frac{c}{\beta-1} \right]^n
\]
is nonempty and this number does not depend on $\ell$.
Consequently, if $\mathcal{F}$ is the family of sets
\[
	(K_j + \mathbf{x}) \cap \left[\frac{-c}{\beta -1},\frac{c}{\beta-1} \right]^n, \quad \mathbf{x} 
	\in \mathbf{D}[\beta], \, j \in \{1,\dots,m\},
\]
then $\mathcal{F}$ is finite and any set $N_{\ell,\mathbf{b}}$ is a finite union of elements of $\mathcal{F}$. 
The $(\beta,C)$-kernel of $X$ is thus finite.
\end{proof}

\begin{theorem} \label{thm:self-similar implique digraph}
Let $\beta>1$ be a real number and $C$ be the alphabet $\{-c,-c+1,\dots,0,\dots,c\}$ for $c \in \mathbb{Z}_{\geq 1}$. 
If $X \subset \left[\frac{-c}{\beta -1},\frac{c}{\beta -1} \right]^n$ is $(\beta,C)$-self-similar, 
then $X$ belongs to the attractor of a GDIFS whose similarities are of the form 
$S_\mathbf{a}(\mathbf{x}) = \frac{\mathbf{x}+\mathbf{a}}{\beta}$ for $\mathbf{a} \in \mathbf{C}$.
\end{theorem}

\begin{proof}
Let
\[
	\mathcal{N}(X) = \{N_{k_i,\mathbf{b}_i}(X) \mid i \in \{1,\dots,m\}, k_i 
					\in \mathbb{Z}_{\geq 0}, \mathbf{b}_i \in \mathbf{C}^{< k_i}[\beta] \}
\]
be the $(\beta,C)$-kernel of $X$. We will build a GDIFS with $m$ vertices, whose attractor is $\mathcal{N}(X)$. 
This will prove the result since $X$ belongs to $\mathcal{N}(X)$: $X= N_{0,\mathbf{0}}(X)$.

For all $k \in \mathbb{Z}_{\geq 0}$ and all $\mathbf{b} \in \mathbf{C}^{<k}[\beta]$, 
we let $f_{k,\mathbf{b}}$ denote the function $\mathbf{x} \mapsto \frac{\mathbf{x}+\mathbf{b}}{\beta^k}$. 
It comes
\[
	N_{k,\mathbf{b}}(X) = f_{k,\mathbf{b}}^{-1}(X) \cap \left[\frac{-c}{\beta -1},\frac{c}{\beta-1} \right]^n
\]
and, for all $k_1,k_2,\mathbf{b}_1,\mathbf{b}_2$,
\[
	f_{k_1 + k_2, \beta^{k_2} \mathbf{b}_1 + \mathbf{b}_2}^{-1} = f_{k_2,\mathbf{b}_2}^{-1} \circ f_{k_1,\mathbf{b}_1}^{-1}
\]
Thus, for all $k \in \mathbb{Z}_{\geq 0}$ and all $\mathbf{b} \in \mathbf{C}^{<k}[\beta]$, we have
\begin{eqnarray*}
	\bigcup_{{a} \in \mathbf{C}} f_{1,\mathbf{a}}\big (N_{k+1,\beta \mathbf{b}+\mathbf{a}}(X)\big) 
		& = &	\bigcup_{\mathbf{a} \in \mathbf{C}} f_{1,\mathbf{a}}\left(f_{k+1,\beta \mathbf{b}+\mathbf{a}}^{-1}(X) 
				\cap \left[\frac{-c}{\beta -1},\frac{c}{\beta-1} \right]^n \right)	\\
		& = &	\bigcup_{\mathbf{a} \in \mathbf{C}} f_{k,\mathbf{b}}^{-1}(X) 
				\cap f_{1,\mathbf{a}}\left(\left[\frac{-c}{\beta -1},\frac{c}{\beta-1} \right]^n \right)	\\
		& = &	f_{k,\mathbf{b}}^{-1}(X) \cap \bigcup_{\mathbf{a} \in \mathbf{C}} f_{1,\mathbf{a}}
				\left(\left[\frac{-c}{\beta -1},\frac{c}{\beta-1} \right]^n \right)	\\
		& = &	f_{k,\mathbf{b}}^{-1}(X) \cap \left[\frac{-c}{\beta -1},\frac{c}{\beta-1} \right]^n 	\\
		& = & 	N_{k,\mathbf{b}}(X)
\end{eqnarray*}
To complete the proof, we build the following GDIFS whose attractor is $\mathcal{N}(X)$:
\begin{enumerate}
	\item	the vertices of the directed graph are $v_1,\dots,v_m$, each $v_i$ corresponding to $N_{k_i,\mathbf{b}_i} (X)$;
	\item	for all vertices $v_i$ and $v_j$, there is an edge from $v_i$ to $v_j$ 
			if there is a letter $\mathbf{a} \in \mathbf{C}$ such that 
			$N_{k_j,\mathbf{b}_j} = N_{k_i+1,\beta \mathbf{b}_i+\mathbf{a}}$. 
			This edge is labeled by $f_{1,\mathbf{a}} = S_\mathbf{a}$.
\end{enumerate}
\end{proof}

We show that Conjecture~\ref{conjecture:AB} follows from
Theorem~\ref{thm cobham version weak buchi}, 
Theorem~\ref{thm:digraph automate}
and Theorem~\ref{thm:self-similar implique digraph}. 

First, notice that rational polyhedra and sets definable in $\langle\mathbb{R},\mathbb{Z},+,\le\rangle$ are two equivalent notions. More precisely, 
the theory $\langle\mathbb{R},\mathbb{Z},+,\le\rangle$ admits the elimination of quantifiers \cite{Ferrante&Rackoff}. 
As a consequence of this result, a set is definable by a first order formula in $\langle\mathbb{R},\mathbb{Z},+,\le\rangle$ 
if and only if it is a finite union of polyhedra whose vertices have rational coordinates. 
Indeed, such a set can be expressed as a finite Boolean combination of linear constraints with rational coefficients.

Let $b\ge 2$ be an integer. It is clear that  a finite union of polyhedra whose vertices have rational coordinates is $b$-self-similar. 

Let $b,b'\ge 2$ be two multiplicatively independent integers. 
Let $X\subset[0,1]^n$ be a compact set that is both $b$- and $b'$-self-similar. 
Applying first  Theorem~\ref{thm:self-similar implique digraph} and then Theorem~\ref{thm:digraph automate}, 
the base~$b$ expansions of the elements in $X$ are accepted by a trim B{\"u}chi automaton that is closed. 
From Remark~\ref{rem:closedweak} this automaton is also weak 
and we can therefore use Theorem~\ref{thm cobham version weak buchi}.

\section{Further work}

To conclude the paper, let us describe what are the constructions for which an analogue of Cobham's theorem exists. 

In Figure~\ref{fig: general picture}, the small ellipse concerns IFS 
and the analogue of Cobham's theorem is Theorem~\ref{thm: IFS} for homogeneous IFSs in $\mathbb{R}$. 
As explained in the introduction of the paper, this result can been extended to non-homogeneous IFSs 
satisfying the strong separation condition in $\mathbb{R}^n$~\cite{EKM:2010}.

In the small rectangle, an analogue of Cobham's theorem exists in $\mathbb{R}^n$ for $\beta \in \mathbb{Z}$: 
it is Theorem~\ref{thm cobham version weak buchi}.
Another one also exists in the big ellipse but the condition of independence of the integer bases is stronger: 
$b$ and $b'$ cannot share common prime factors~\cite{BBB10}.

There exist sets that do not belong to any of these cases, for instance the Rauzy fractal (Example~\ref{ex:Rauzy fractal}).
Let $\alpha$ be a complex root of the polynomial $x^3 - x^2 - x - 1$.
The Rauzy fractal is the attractor of a non-homogeneous IFS (contraction ratios are different powers of $\alpha$) 
not satisfying the strong separation condition.
It is also the attractor of a homogeneous GDIFS with contraction ratio $\alpha$. 
Finally, it is the set of complex number 
$\left\{\sum_{i=3}^{+\infty} u_i \alpha_i \mid \forall i \geq 3, \, u_i \in \{0,1\} \text{ and } u_iu_{i+1}u_{i+2} \neq 111\right\}$; 
the $\alpha$-representation of its elements are therefore accepted by some Büchi automaton.
Since these constructions do not meet the settings where an analogue of Cobham's theorem is known, 
there currently exists no result attesting that this set cannot be obtained in another way 
(through an IFS, a GDIFS or an automata for instance) with some contraction ratio ``independent" from $\alpha$.

\section{Acknowledgement}
We would like to thank Fabien Durand for many interesting discussions.
This research is supported by the internal research project F1R-MTH-PUL-12RDO2 of the University of Luxembourg.

\bibliographystyle{alpha}
\bibliography{biblio3}

\end{document}